\documentclass[pdflatex,sn-mathphys-num]{sn-jnl}

\usepackage{amsmath,amssymb,amsthm,mathtools,amsfonts}
\usepackage{graphicx}
\usepackage{multirow}
\usepackage{mathrsfs}
\usepackage[title]{appendix}
\usepackage{xcolor}
\usepackage{textcomp}
\usepackage{manyfoot}
\usepackage{booktabs}
\usepackage{algorithm}
\usepackage{lineno}
\usepackage{algorithmicx}
\usepackage{algpseudocode}
\usepackage{listings}
\usepackage{hyperref}
\usepackage{enumitem}
\usepackage[mathscr]{eucal}
\usepackage{tikz}
\usetikzlibrary{arrows.meta, positioning, calc, decorations.pathmorphing}

\hypersetup{
  colorlinks=true,
  linkcolor=blue!55!black,
  citecolor=blue!55!black,
  urlcolor=blue!55!black
}

\DeclareMathOperator*{\argmax}{arg\,max}
\DeclareMathOperator*{\argmin}{arg\,min}
\DeclareMathOperator{\dom}{dom}
\DeclareMathOperator{\intr}{int}

\newcommand{\R}{\mathbb{R}}
\newcommand{\Rn}{\mathbb{R}^n}
\newcommand{\Rm}{\mathbb{R}^m}
\newcommand{\U}{\mathcal{U}}
\newcommand{\Cc}{\mathcal{C}}
\newcommand{\Ht}{\widetilde H}
\newcommand{\tp}{^\top}
\newcommand{\ip}[2]{\langle #1,#2\rangle}

\newtheorem{theorem}{Theorem}[section]
\newtheorem{proposition}[theorem]{Proposition}
\newtheorem{lemma}[theorem]{Lemma}
\newtheorem{corollary}[theorem]{Corollary}
\theoremstyle{definition}
\newtheorem{definition}[theorem]{Definition}
\newtheorem{assumption}[theorem]{Assumption}
\newtheorem{remark}[theorem]{Remark}

\numberwithin{equation}{section}

\usepackage{tabularx}

\begin{document}

\title{\bf Multiplier Sensitivity in Isoperimetric Optimal Control}

\author*[1]{\fnm{Ye} \sur{Liang}}
\email{ye-liang@uiowa.edu}

\affil[1]{\orgdiv{College of Engineering},
\orgname{The University of Iowa},
\orgaddress{
\city{Iowa City}, 
\state{IA},
\postcode{52242},      
\country{United States}
}}

\abstract{
We study finite-horizon optimal control problems with scalar isoperimetric constraints from a function-space duality perspective. Controls are treated as elements of weakly compact subsets of \(L^\infty(0,T;\mathbb R^m)\), while the state equation induces a control-to-state map into \(W^{1,\infty}(0,T;\mathbb R^n)\). For linear dynamics, concave payoff, and an affine isoperimetric functional, we prove that the constrained value function has an interval domain, is concave, and admits a Fenchel--Moreau dual
representation. This yields a superdifferential formula identifying the
negative of the dual multiplier with the sensitivity of the value function with
respect to the constraint level. A constraint qualification is then used to
identify the dual multiplier with the normal Pontryagin multiplier of the
augmented isoperimetric system. We also treat linear-quadratic problems with a
single quadratic equality constraint by reducing them to quadratic forms on a
Hilbert space. The resulting analysis separates the validity of the envelope
formula from the regularity needed for Riccati synthesis, showing that
sensitivity may persist even when the modified control-weight operator becomes
singular.}



\maketitle 

\section{Introduction and Main Contributions}\label{sec:intro}

Optimal control problems with integral constraints, often called
\emph{isoperimetric constraints}, arise in resource-limited control,
energy-budgeted optimization, economic growth, and constrained engineering
design. In such problems the admissible controls must satisfy both a state
equation and a prescribed accumulated quantity,
\[
    \int_0^T h(t,x(t),u(t))\,dt = B .
\]
The classical Pontryagin maximum principle treats this constraint by
augmenting the state. Introducing
\[
    z(t)=\int_0^t h(s,x(s),u(s))\,ds,
\]
one obtains
    $\dot z(t)=h(t,x(t),u(t))$, $z(0)=0$, $z(T)=B$,
and the endpoint condition \(z(T)=B\) gives rise to a scalar adjoint multiplier
\(\mu\) associated with the isoperimetric constraint.

The present paper studies the function-space and duality mechanisms that determine
when this scalar multiplier has a genuine sensitivity interpretation. The
admissible controls are viewed as elements of weakly compact subsets of
\(L^\infty(0,T;\mathbb R^m)\), while the state equation defines a
control-to-state map into \(W^{1,\infty}(0,T;\mathbb R^n)\). In the structured
linear case this map is affine and weak-\(*\) continuous, and therefore the
isoperimetric problem becomes an optimization problem over a Banach function
space with a scalar constraint functional.

Let
\[
    \Cc(u):=\int_0^T h(t,x_u(t),u(t))\,dt
\]
denote the accumulated constraint functional, where \(x_u\) is the trajectory
generated by the control \(u\), and define the constrained value function
    $V(B):=\sup\bigl\{J(u):u\in\U,\ \Cc(u)=B\bigr\}$.
The central question is whether the Pontryagin multiplier \(\mu\) can be
identified with the marginal value of changing \(B\). Schematically, one seeks
\[
    \text{Pontryagin multiplier}
    \quad\Longrightarrow\quad
    \text{shadow value of the isoperimetric level}.
\]
This implication is not automatic. For a nonlinear equality-constrained
control problem, the maximum principle alone does not imply that \(V\) is
concave, locally Lipschitz, or differentiable. Equality-constrained attainable
sets may be nonconvex, and different constraint levels may correspond to
disconnected families of feasible or optimal controls. Thus the existence of
an isoperimetric multiplier is a first-order necessary condition, not by itself
an envelope theorem.

The main theme of the paper is that the shadow-price interpretation is a
duality statement in function spaces. In the maximization setting, the relevant
dual representation has the form
    $V(B)=\inf_{\mu\in\mathbb R}\{\ell(\mu)-\mu B\}$,
    $\ell(\mu):=\sup_{u\in\U}\{J(u)+\mu\Cc(u)\}$.
When this representation is valid, Fenchel--Moreau duality identifies the
superdifferential of \(V\) with the negative of the set of dual-optimal
multipliers. A further constraint qualification is then needed to identify
these dual multipliers with the normal Pontryagin multipliers of the augmented
isoperimetric system. This separation is important: the convex-analytic
superdifferential identity is a property of the value function, whereas the
identification with Pontryagin multipliers is a normality and multiplier
representation statement.

\paragraph{Main contributions.}
The paper develops this program in two function-space regimes.

\begin{enumerate}[label=\textup{(C\arabic*)}]
  \item \textbf{Weak-\(*\) compactness and affine isoperimetric constraints.}
  For linear dynamics, concave payoff, compact convex pointwise control set,
  and affine isoperimetric integrand, the control-to-state map is affine and
  weak-\(*\) continuous from \(L^\infty(0,T;\mathbb R^m)\) into
  \(C([0,T];\mathbb R^n)\). We prove that the attainable set
  \(\dom V=\Cc(\U)\) is an interval and that the constrained value function is
  concave. Under a two-sided attainability condition, Fenchel--Moreau duality
  yields the strong dual representation of \(V\). Thus the dual representation
  is derived from primitive function-space hypotheses rather than assumed.

  \item \textbf{Superdifferential sensitivity by convex duality.}
  Once \(V\) is finite and concave on the interior of its domain, the identity
      $\partial^+V(B)=-\mathcal D(B)$
  follows from conjugate duality, where \(\mathcal D(B)\) denotes the set of
  dual-optimal multipliers at level \(B\). Hence, at every differentiability
  point,
      $V'(B)=-\mu$.
  The additional step is to identify \(\mathcal D(B)\) with the set
  \(\mathcal M(B)\) of normal Pontryagin isoperimetric multipliers. We prove
  this under an explicit constraint qualification based on two-sided
  attainability of the constraint level.

  \item \textbf{Hilbert-space quadratic forms and LQ constraints.}
  For linear-quadratic problems with a single quadratic equality constraint,
  eliminating the state reduces the problem to quadratic forms on the Hilbert
  space \(L^2(0,T;\mathbb R^m)\):
  \[
      J_{\rm LQ}(u)
      =
      \frac12\langle u,\mathcal R u\rangle
      +
      \langle \rho_0,u\rangle+\gamma_0,
      \qquad
      \Cc(u)=\langle u,\mathcal D u\rangle .
  \]
  This exposes a second hidden-convexity mechanism, related to the joint range
  of two quadratic functionals. The envelope formula remains a duality
  statement and can persist even when the modified control-weight operator is
  singular. By contrast, Riccati synthesis requires the stronger regularity
  condition
      $R(t)+2\mu^*D(t)\succ0$ a.e.
  at the optimal multiplier. The degenerate scalar example shows that
  sensitivity may survive precisely at the boundary where Riccati feedback
  fails.
\end{enumerate}

The main conclusion is therefore
    $-\mu\in\partial^+V(B)$,
and, whenever \(V\) is differentiable at \(B\),
    $V'(B)=-\mu$.
Thus the negative of the isoperimetric multiplier
is the marginal value of relaxing the constraint level under the Lagrangian
convention \(J(u)+\mu(\Cc(u)-B)\).

\paragraph{Organization.}
Section~\ref{sec:preliminaries} formulates the control problem on weak
function spaces, proves well-posedness and existence under compactness and
closedness hypotheses, and recalls the normal-form maximum principle.
Section~\ref{sec:sensitivity} develops the isoperimetric augmentation,
establishes weak-\(*\) compactness and strong duality for the structured
concave-affine class, and proves the superdifferential representation of the
value function. Section~\ref{sec:lq} treats linear-quadratic isoperimetric
problems as Hilbert-space quadratic programs, separating the dual sensitivity
formula from the additional regularity required for Riccati synthesis.
Section~\ref{sec:numerics} gives a shooting implementation and numerical
checks of the envelope identity.

The analysis is based on weak compactness, convex duality, and Hilbert-space
quadratic form theory rather than on dynamic programming. This functional-
analytic viewpoint is the reason the isoperimetric multiplier can be treated
as a shadow price only after the appropriate duality and normality mechanisms
have been established.

\paragraph{Related work.}
The maximum principle and its isoperimetric augmentation are classical
\cite{pontryagin2018mathematical,hestenes1966calculus,gamkrelidze2013principles,berkovitz2013optimal,lee1967foundations,bryson2018applied,athans2013optimal,wang2025analysis,
fleming2012deterministic,vinter2010optimal,clarke2013functional,liberzon2011calculus,ioe1979theory}, as are the underlying
calculus-of-variations problems with integral constraints
\cite{bliss1946lectures,gelfand1963calculus,wang2025analysis1,giaquinta2013calculus}. Existence in the constrained
setting rests on convexity/relaxation and lower-closure theory
\cite{cesari2012optimization,warga2014optimal,liu2025bidirectional,young2024lectures,balder1984general,ekeland1999convex}. Our sensitivity analysis is in
the tradition of perturbation and value-function theory in optimization
\cite{rockafellar1997convex,rockafellar1974conjugate,liang2025global,rockafellar1998variational,bonnans2013perturbation,fiacco1983introduction,gauvin1977necessary,
mordukhovich2006variational1,mordukhovich2006variational2,aubin1990set,wang2026algebraic,cannarsa2004semiconcave,luenberger1997optimization,wang2025multi,gao2022rolling,
borwein2006convex}, and the identification of Pontryagin with dual multipliers uses
constraint-qualification and multiplier theory in Banach spaces
\cite{robinson1976stability,zowe1979regularity,wang2026damage,kurcyusz1976existence,maurer1979first,jahn2007introduction}. The linear-quadratic
results draw on the convexity of the joint range of quadratic forms
\cite{toeplitz1918algebraische,hausdorff1919wertvorrat,dines1941mapping,yu2026pattern,brickman1961field,polyak1998convexity,hiriart2002permanently}, on the
S-procedure and its duality theory
\cite{yakubovich1971s,fradkov1979thes,polik2007survey,ben2001lectures,beck2006strong}, on
the trust-region subproblem and its hard case
\cite{gay1981computing,sorensen1982newton,more1983computing,wang2026breakdown,conn2000trust,stern1995indefinite,martinez1994local}, and on
Riccati theory and integral quadratic constraints
\cite{anderson2007optimal,wonham1968matrix,bittanti1991riccati,liu2026computational,lancaster1995algebraic,abou2012matrix,megretski1997system,yu2026rigorous,
willems1972dissipative}.

The paper is organized as follows.
Section~\ref{sec:preliminaries} states the standing assumptions and the
normal-form maximum principle, and proves existence for both the unconstrained
and the structured constrained problem. Section~\ref{sec:sensitivity} develops
the augmented maximum principle, proves concavity and strong duality for the
structured class, establishes the conjugate-duality characterization of the
superdifferential, and adds the constraint-qualification step identifying the
multiplier sets. Section~\ref{sec:lq} treats linear-quadratic isoperimetric
problems, proving hidden convexity via Dines--Polyak and separating sensitivity
from Riccati synthesis, with the degenerate case worked out explicitly.

\section{Problem Formulation, Existence, and Normal PMP}
\label{sec:preliminaries}

Let $T>0$ be fixed. We consider
\begin{equation}\label{eq:unconstrained-problem}
\begin{aligned}
    &\text{maximize} &&
    J(u):=\int_0^T f(t,x(t),u(t))\,dt+\phi(x(T))\\
    &\text{subject to} &&
    \dot x(t)=g(t,x(t),u(t)),\qquad x(0)=x_0,\qquad u\in\U .
\end{aligned}
\end{equation}
Here $x(t)\in\Rn$, $u(t)\in\Rm$, and $x_0\in\Rn$ is fixed. The admissible
controls are
\[
\U:=\bigl\{u\in L^\infty(0,T;\Rm): u(t)\in U \text{ for a.e. }t\in[0,T]\bigr\},
\]
with $U\subset\Rm$ a prescribed control set.

\subsection{Standing assumptions}

\begin{assumption}[State equation and payoff]\label{ass:data}
The maps $g:[0,T]\times\Rn\times\Rm\to\Rn$, $f:[0,T]\times\Rn\times\Rm\to\R$,
and $\phi:\Rn\to\R$ satisfy:
\begin{enumerate}[label=\textup{(A\arabic*)}]
  \item $t\mapsto g(t,x,u)$ and $t\mapsto f(t,x,u)$ are measurable for each
  $(x,u)$.
  \item $(x,u)\mapsto g(t,x,u)$ and $(x,u)\mapsto f(t,x,u)$ are $C^1$ for a.e.
  $t$.
  \item There is $L>0$ with
  $\|g(t,x_1,u)-g(t,x_2,u)\|\le L\|x_1-x_2\|$ for all $x_1,x_2$, all $u\in U$,
  a.e.\ $t$.
  \item There is $c>0$ with $\|g(t,x,u)\|\le c(1+\|x\|+\|u\|)$ and
  $|f(t,x,u)|\le c(1+\|x\|^2+\|u\|^2)$ on $[0,T]\times\Rn\times U$.
  \item $\phi\in C^1(\Rn)$ and $|\phi(x)|\le c(1+\|x\|^2)$.
\end{enumerate}
\end{assumption}

\begin{assumption}[Compactness and closedness]\label{ass:compact-closed}
The set $U\subset\Rm$ is nonempty, compact, and convex. Moreover the feasible
trajectory set is sequentially closed: if $u_j\rightharpoonup^* u$ in
$L^\infty(0,T;\Rm)$ with $u_j(t)\in U$ a.e., and $x_j\to x$ uniformly on
$[0,T]$, then $x$ is the trajectory of $u$; and along every such maximizing
sequence $\limsup_j J(u_j)\le J(u)$.
\end{assumption}

\begin{remark}[Role of closedness, and the structured class]\label{rem:closedness}
Assumption~\ref{ass:compact-closed} is stated explicitly to avoid hiding a
compactness argument. Its substantive content is the nonlinear limit passage
$\dot x_j=g(t,x_j,u_j)\to\dot x=g(t,x,u)$, which for $g$ nonlinear in $u$
requires a convexity/relaxation structure. In the \emph{structured class} of
Section~\ref{sec:sensitivity} (linear dynamics, concave payoff, convex $U$),
this closedness is automatic: the control-to-state map is affine and weak-$*$
continuous, and $J$ is weak-$*$ upper semicontinuous. The same structural
hypotheses drive the duality results below, so the paper's standing assumptions
are unified.
\end{remark}

\subsection{Well-posedness and existence}

\begin{proposition}[Existence and uniqueness of trajectories]\label{prop:wellposed}
Under Assumption~\ref{ass:data}, for every $u\in\U$ the initial value problem
$\dot x(t)=g(t,x(t),u(t))$, $x(0)=x_0$,
has a unique solution $x_u\in W^{1,\infty}(0,T;\Rn)$.
\end{proposition}

\begin{proof}
Fix $u\in\U$. The map $(t,x)\mapsto g(t,x,u(t))$ is measurable in $t$, $C^1$
in $x$ a.e., and uniformly Lipschitz in $x$; Carath\'eodory's theorem gives a
local absolutely continuous solution. The growth bound gives
$\|\dot x(t)\|\le c_u(1+\|x(t)\|)$ with $c_u=c(1+\|u\|_{L^\infty})<\infty$, so
Gronwall's inequality precludes blow-up and extends the solution to $[0,T]$.
Two solutions $x_1,x_2$ satisfy
\[
\|x_1(t)-x_2(t)\|\le L\int_0^t\|x_1-x_2\|\,ds,
\]
whence uniqueness by Gronwall.
Boundedness of $x_u$ and the growth bound give $\dot x_u\in L^\infty$, so
$x_u\in W^{1,\infty}$.
\end{proof}

\begin{theorem}[Existence for the unconstrained problem]\label{thm:existence}
Under Assumptions~\ref{ass:data} and~\ref{ass:compact-closed}, problem
\eqref{eq:unconstrained-problem} admits an optimal control.
\end{theorem}

\begin{proof}
Let $\{u_j\}\subset\U$ be maximizing. By compactness of $U$ the sequence is
bounded in $L^\infty$; by Banach--Alaoglu, along a subsequence
$u_j\rightharpoonup^* u^*$. Since $U$ is closed and convex the pointwise
constraint is weak-$*$ closed, so $u^*\in\U$. By
Proposition~\ref{prop:wellposed} the states $x_j$ are bounded in $W^{1,\infty}$,
so by Arzel\`a--Ascoli $x_j\to x^*$ uniformly along a further subsequence. By
Assumption~\ref{ass:compact-closed}, $x^*$ is the trajectory of $u^*$ and
$J(u^*)\ge\limsup_j J(u_j)=\sup_\U J$. Hence $u^*$ is optimal.
\end{proof}

\subsection{Normal-form Pontryagin maximum principle}

Define the Hamiltonian $H(t,x,u,\lambda):=f(t,x,u)+\lambda\tp g(t,x,u)$.

\begin{theorem}[Normal-form necessary conditions]\label{thm:pmp-normal}
Let Assumption~\ref{ass:data} hold and let $u^*\in\U$ be a normal local
maximizer of \eqref{eq:unconstrained-problem}, with trajectory $x^*=x_{u^*}$.
Then there exists $\lambda\in W^{1,1}(0,T;\Rn)$ such that, a.e.\ on $[0,T]$,
\begin{align}
  \dot x^*(t)&=g(t,x^*(t),u^*(t)),\label{eq:pmp-state}\\
  \dot\lambda(t)&=-H_x(t,x^*(t),u^*(t),\lambda(t)),\label{eq:pmp-adjoint}\\
  H(t,x^*(t),u^*(t),\lambda(t))&=\max_{u\in U}H(t,x^*(t),u,\lambda(t)),
  \label{eq:pmp-max}
\end{align}
with the transversality condition
\begin{equation}\label{eq:pmp-terminal}
  \lambda(T)=\nabla\phi(x^*(T)).
\end{equation}
\end{theorem}

\begin{remark}[Normality]\label{rem:normality}
In the general principle the payoff carries a multiplier $\lambda_0\ge0$, with
the normal case $\lambda_0=1$. Abnormal extremals ($\lambda_0=0$) are excluded
here because the sensitivity theory interprets the isoperimetric multiplier as a
shadow price. Section~\ref{sec:sensitivity} makes the required constraint
qualification explicit.
\end{remark}

\begin{remark}[Maximization convention]\label{rem:max-convention}
Problem \eqref{eq:unconstrained-problem} is a maximization, so
\eqref{eq:pmp-max} is a pointwise maximization. The linear-quadratic examples of
Section~\ref{sec:lq} are formulated as minimization problems, and the extremality
condition there is written in minimization form.
\end{remark}

\begin{proposition}[Mangasarian sufficiency]\label{prop:mangasarian}
Let $u^*,x^*,\lambda$ satisfy \eqref{eq:pmp-state}--\eqref{eq:pmp-terminal}.
If, for a.e.\ $t$ and fixed $\lambda=\lambda(t)$, the map $(x,u)\mapsto
H(t,x,u,\lambda)$ is concave on $\Rn\times U$ and $\phi$ is concave, then $u^*$
is a global maximizer.
\end{proposition}

\begin{proof}
For arbitrary $u\in\U$ with $x=x_u$, using
$f=H-\lambda\tp g$ and $\dot x=g$,
\[
J(u^*)-J(u)=\int_0^T\!\bigl[H(t,x^*,u^*,\lambda)-H(t,x,u,\lambda)
-\lambda\tp(\dot x^*-\dot x)\bigr]dt+\phi(x^*(T))-\phi(x(T)).
\]
Since $x^*(0)=x(0)$, integration by parts gives
$-\int_0^T\lambda\tp(\dot x^*-\dot x)\,dt
=\int_0^T\dot\lambda\tp(x^*-x)\,dt-\lambda(T)\tp(x^*(T)-x(T))$.
Substituting $\dot\lambda=-H_x(t,x^*,u^*,\lambda)$ and
$\lambda(T)=\nabla\phi(x^*(T))$,
\[
J(u^*)-J(u)=\int_0^T\!\bigl[H(t,x^*,u^*,\lambda)-H(t,x,u,\lambda)
-H_x(t,x^*,u^*,\lambda)\tp(x^*-x)\bigr]dt+\Delta_\phi,
\]
with $\Delta_\phi=\phi(x^*(T))-\phi(x(T))-\nabla\phi(x^*(T))\tp(x^*(T)-x(T))\ge0$
by concavity of $\phi$. The integrand is nonnegative by concavity of $H$ and the
maximization \eqref{eq:pmp-max}. Hence $J(u^*)\ge J(u)$.
\end{proof}

\section{Isoperimetric Constraints and Dual Sensitivity}
\label{sec:sensitivity}

We now introduce a scalar equality constraint and study the sensitivity of the
constrained value function. Throughout,
  $\Cc(u):=\int_0^T h(t,x_u(t),u(t))\,dt$,
and, for $B\in\R$,
\begin{equation}\label{eq:constrained-value}
  V(B):=\sup\bigl\{J(u):u\in\U,\ \Cc(u)=B\bigr\},
\end{equation}
with $V(B)=-\infty$ if $B$ is infeasible, so
$\dom V=\{\Cc(u):u\in\U\}$.

\begin{assumption}[Isoperimetric integrand]\label{ass:h}
$h:[0,T]\times\Rn\times\Rm\to\R$ is measurable in $t$, $C^1$ in $(x,u)$ a.e.,
and satisfies $|h(t,x,u)|\le c_h(1+\|x\|^2+\|u\|^2)$ on
$[0,T]\times\Rn\times U$.
\end{assumption}

\subsection{Augmented formulation and augmented PMP}

Writing $z(t)=\int_0^t h(s,x(s),u(s))\,ds$ gives $\dot z=h$, $z(0)=0$, and the
endpoint condition $z(T)=B$. The augmented Hamiltonian is
\begin{equation}\label{eq:augmented-hamiltonian}
  \Ht(t,x,u,\lambda,\mu):=f(t,x,u)+\lambda\tp g(t,x,u)+\mu\,h(t,x,u),
\end{equation}
with $\mu\in\R$ the adjoint of $z$. Since $\Ht$ is independent of $z$, the
$z$-adjoint equation forces $\mu$ to be constant.

\begin{theorem}[Normal PMP for the isoperimetric problem]\label{thm:pmp-isoperimetric}
Assume Assumptions~\ref{ass:data},~\ref{ass:compact-closed},~\ref{ass:h}. Let
$B\in\dom V$ and let $u^*$ be a normal local maximizer for $V(B)$, with
$x^*=x_{u^*}$. Then there exist $\lambda\in W^{1,1}(0,T;\Rn)$ and a constant
$\mu\in\R$ such that, a.e.\ on $[0,T]$,
\begin{align}
  \dot x^*(t)&=g(t,x^*(t),u^*(t)),\label{eq:iso-pmp-state}\\
  \dot\lambda(t)&=-\nabla_x\Ht(t,x^*(t),u^*(t),\lambda(t),\mu),
  \label{eq:iso-pmp-adjoint}\\
  \Ht(t,x^*(t),u^*(t),\lambda(t),\mu)&=\max_{u\in U}
  \Ht(t,x^*(t),u,\lambda(t),\mu),\label{eq:iso-pmp-max}
\end{align}
with the transversality and feasibility conditions
\begin{align}
  \lambda(T)&=\nabla\phi(x^*(T)),\label{eq:iso-pmp-terminal}\\
  \int_0^T h(t,x^*(t),u^*(t))\,dt&=B.\label{eq:iso-constraint-satisfied}
\end{align}
\end{theorem}

\begin{proof}
Apply the normal-form Pontryagin maximum principle with endpoint equality
constraints to the augmented state $y=(x,z)$ with dynamics $\dot y=(g,h)$, $y(0)=(x_0,0)$, and the prescribed endpoint $z(T)=B$. 
This is the endpoint-constrained version of Theorem~\ref{thm:pmp-normal}; the
fixed terminal condition $z(T)=B$ introduces the constant multiplier $\mu$
rather than a free-endpoint transversality condition for $z$. 
The adjoint splits as $(\lambda,\mu)$; the augmented
Hamiltonian is \eqref{eq:augmented-hamiltonian}. Because $\Ht$ is independent of
$z$, the $z$-adjoint equation reads $\dot\mu=-\partial_z\Ht=0$, so $\mu$ is
constant. The remaining relations are
\eqref{eq:iso-pmp-adjoint}--\eqref{eq:iso-pmp-max} with
$\lambda(T)=\nabla\phi(x^*(T))$.
\end{proof}

\subsection{Affine reduction of the isoperimetric constraint}
\label{subsec:affine-constraint-reduction}
We now show explicitly that, under linear dynamics and an affine
isoperimetric integrand, the constraint functional \(\mathcal C\) is affine in
the control. Suppose
    $\dot x(t)=A(t)x(t)+B(t)u(t)$,
    $x(0)=x_0$,
where \(A\in L^\infty(0,T;\mathbb R^{n\times n})\) and
\(B\in L^\infty(0,T;\mathbb R^{n\times m})\). Let \(\Phi(t,s)\) denote the
state-transition matrix generated by \(A(\cdot)\), i.e.
    $\partial_t\Phi(t,s)=A(t)\Phi(t,s)$, $\Phi(s,s)=I_n$.
Then the trajectory generated by \(u\) is
\begin{equation}\label{eq:variation-of-constants}
    x_u(t)
    =
    \Phi(t,0)x_0
    +
    \int_0^t \Phi(t,s)B(s)u(s)\,ds .
\end{equation}
Assume that the isoperimetric integrand is affine in \((x,u)\):
\begin{equation}\label{eq:affine-integrand}
    h(t,x,u)
    =
    a(t)^\top x+b(t)^\top u+c(t),
\end{equation}
where
    $a\in L^\infty(0,T;\mathbb R^n)$,
    $b\in L^\infty(0,T;\mathbb R^m)$,
    $c\in L^1(0,T)$.
The accumulated constraint functional is
\begin{equation}\label{eq:C-definition-affine}
    \mathcal C(u)
    =
    \int_0^T h(t,x_u(t),u(t))\,dt .
\end{equation}
Substituting \eqref{eq:affine-integrand} into \eqref{eq:C-definition-affine}
gives
\[
    \mathcal C(u)
    =
    \int_0^T
    \left[
        a(t)^\top x_u(t)
        +
        b(t)^\top u(t)
        +
        c(t)
    \right]dt .
\]
Using the representation \eqref{eq:variation-of-constants}, we obtain
\[
    \mathcal C(u)=
    \int_0^T
    a(t)^\top
    \left[
        \Phi(t,0)x_0
        +
        \int_0^t\Phi(t,s)B(s)u(s)\,ds
    \right]dt
    +
    \int_0^T b(t)^\top u(t)\,dt
    +
    \int_0^T c(t)\,dt .
\]
Separating the control-independent and control-dependent terms yields
\begin{equation}\label{eq:C-separated}
    \mathcal C(u)
    =
    \underbrace{
    \int_0^T a(t)^\top\Phi(t,0)x_0\,dt
    +
    \int_0^T c(t)\,dt
    }_{\text{constant part}}
    +
    \int_0^T b(t)^\top u(t)\,dt
    +
    \int_0^T
    a(t)^\top
    \left[
        \int_0^t\Phi(t,s)B(s)u(s)\,ds
    \right]dt .
\end{equation}
We now rewrite the double integral as a single integral in \(u\). First,
\[
\begin{aligned}
    I(u) :=
    \int_0^T
    a(t)^\top
    \left[
        \int_0^t\Phi(t,s)B(s)u(s)\,ds
    \right]dt =
    \int_0^T
    \left[
        \int_s^T
        a(t)^\top\Phi(t,s)B(s)
        \,dt
    \right]
    u(s)\,ds .
\end{aligned}
\]
Using the identity
    $a(t)^\top\Phi(t,s)B(s)
    =
    \left[
        B(s)^\top\Phi(t,s)^\top a(t)
    \right]^\top$,
we may write
\begin{equation}\label{eq:double-integral-vector-form}
\begin{aligned}
    I(u)
    &=
    \int_0^T
    \left[
        \int_s^T
        B(s)^\top\Phi(t,s)^\top a(t)\,dt
    \right]^\top
    u(s)\,ds .
\end{aligned}
\end{equation}
Define the scalar constant
\[
    \delta
    :=
    \int_0^T a(t)^\top\Phi(t,0)x_0\,dt
    +
    \int_0^T c(t)\,dt,
\]
and define the reduced constraint kernel
\[
    \gamma(s)
    :=
    b(s)
    +
    \int_s^T B(s)^\top\Phi(t,s)^\top a(t)\,dt,
    \qquad 0\le s\le T .
\]
Combining \eqref{eq:C-separated} and
\eqref{eq:double-integral-vector-form}, we obtain the affine representation
\[
    \mathcal C(u)
    =
    \delta
    +
    \int_0^T \gamma(s)^\top u(s)\,ds .
\]
Equivalently, using the \(L^2\)-duality pairing,
    $\mathcal C(u)
    =
    \delta+\langle \gamma,u\rangle_{L^2(0,T;\mathbb R^m)}$.
Thus \(\mathcal C\) is an affine functional of \(u\).
For later use, the reduced equality constraint \(\mathcal C(u)=B_0\) can be
written as
    $\langle\gamma,u\rangle_{L^2}
    +
    \delta
    -
    B_0
    =
    0$.
Hence the reduced constraint map is
    $\widehat G(u)
    :=
    \mathcal C(u)-B_0
    =
    \langle\gamma,u\rangle_{L^2}
    +
    \delta
    -
    B_0$.
Its Fréchet derivative is the constant linear functional
    $D\widehat G(u)v
    =
    \langle\gamma,v\rangle_{L^2(0,T;\mathbb R^m)}$
    for all $u,v\in L^2(0,T;\mathbb R^m)$.
Therefore,
    $\widehat G(u+v)
    =
    \widehat G(u)
    +
    D\widehat G(u)v
    =
    \widehat G(u)
    +
    \langle\gamma,v\rangle_{L^2}$.
In particular, for two controls \(u_0,u_1\in\mathcal U\) and
\(\theta\in[0,1]\), setting
    $u_\theta=(1-\theta)u_0+\theta u_1$,
we obtain
\[
\begin{aligned}
    \mathcal C(u_\theta)
    &=
    \delta
    +
    \int_0^T
    \gamma(t)^\top
    \bigl[(1-\theta)u_0(t)+\theta u_1(t)\bigr]dt
    \\
    &=
    (1-\theta)
    \left[
        \delta+\int_0^T\gamma(t)^\top u_0(t)\,dt
    \right]
    +
    \theta
    \left[
        \delta+\int_0^T\gamma(t)^\top u_1(t)\,dt
    \right]
    \\
    &=
    (1-\theta)\mathcal C(u_0)+\theta\mathcal C(u_1).
\end{aligned}
\]
Consequently, if \(\mathcal C(u_0)=B_0\) and \(\mathcal C(u_1)=B_1\), then
\begin{equation}\label{eq:level-interpolation}
    \mathcal C(u_\theta)
    =
    (1-\theta)B_0+\theta B_1.
\end{equation}
This is the precise algebraic reason why, in the structured affine-constraint
class, the attainable set of isoperimetric levels is an interval and why the
constraint qualification can be expressed as two-sided attainability of the
level \(B_0\).

\subsection{A structured class with provable strong duality}

We now isolate a class in which concavity and strong duality of $V$ are
\emph{theorems} rather than assumptions. The key is that the data make $J$
concave and $\Cc$ affine \emph{as functionals of $u$}, so that the equality
level enters as an affine perturbation.

\begin{assumption}[Structured class]\label{ass:structured}
The dynamics are linear,
\[
  \dot x(t)=A(t)x(t)+B(t)u(t),\qquad x(0)=x_0,
\]
with $A,B$ measurable and bounded; $U\subset\Rm$ is nonempty, compact, convex;
$f(t,\cdot,\cdot)$ is concave and upper semicontinuous and $\phi$ is concave and
upper semicontinuous; and the isoperimetric integrand is affine,
  $h(t,x,u)=a(t)\tp x+b(t)\tp u+c(t)$,
with $a,b,c$ measurable and bounded. Moreover, the reduced payoff 
\[
    J(u):=\int_0^T f(t,x_u(t),u(t))\,dt+\phi(x_u(T))
\]
is weak-$*$ upper semicontinuous on $\U$.
\end{assumption}

\begin{lemma}[Affine control-to-state map]\label{lem:affine-map}
Under Assumption~\ref{ass:structured}, the map $u\mapsto x_u$ is affine and
weak-$*$ continuous from $\U$ to $C([0,T];\Rn)$. Consequently $J$ is concave and
weak-$*$ upper semicontinuous on $\U$, and $\Cc$ is affine and weak-$*$
continuous on $\U$.
\end{lemma}

\begin{proof}
With state transition matrix $\Phi(t,s)$ of $A$,
\[
  x_u(t)=\Phi(t,0)x_0+\int_0^t\Phi(t,s)B(s)u(s)\,ds,
\]
which is affine in $u$; weak-$*$ continuity follows since the kernel
$s\mapsto\Phi(t,s)B(s)$ is bounded, hence in $L^1$. Thus $u\mapsto(x_u,u)$ is
affine. Consequently $J$ is concave on $\U$, and by the weak-$*$ upper
semicontinuity part of Assumption~\ref{ass:structured}, $J$ is weak-$*$
upper semicontinuous. Finally $\Cc(u)=\int_0^T\bigl(a\tp x_u+b\tp u+c\bigr)dt$ is affine in
$u$, and weak-$*$ continuous.
\end{proof}

\begin{theorem}[Concavity and interval domain]\label{thm:struct-concave}
Under Assumption~\ref{ass:structured}, $\dom V$ is an interval and $V$ is
concave and upper semicontinuous on $\dom V$.
\end{theorem}

\begin{proof}
Since $\Cc$ is affine and $\U$ is convex, $\dom V=\Cc(\U)$ is the affine image
of a convex set, hence an interval. Let $B_0,B_1\in\dom V$, $\theta\in[0,1]$,
$B_\theta=(1-\theta)B_0+\theta B_1$. For $\varepsilon>0$ choose feasible
$u_0,u_1$ with $\Cc(u_i)=B_i$ and $J(u_i)\ge V(B_i)-\varepsilon$. Put
$u_\theta=(1-\theta)u_0+\theta u_1\in\U$ (convexity of $\U$). By \eqref{eq:level-interpolation}, $u_\theta$ is feasible for $B_\theta$, and by concavity of $J$,
  $V(B_\theta)\ge J(u_\theta)\ge(1-\theta)J(u_0)+\theta J(u_1)
  \ge(1-\theta)V(B_0)+\theta V(B_1)-\varepsilon$.
Letting $\varepsilon\downarrow0$ gives concavity. Upper semicontinuity of $V$
follows from weak-$*$ upper semicontinuity of $J$, weak-$*$ closedness of $\U$,
and weak-$*$ continuity of $\Cc$ (Lemma~\ref{lem:affine-map}) by the argument of
Theorem~\ref{thm:existence} applied at fixed level $B$.
\end{proof}

The next result is the constraint qualification. It plays the role of a Slater
condition and coincides with two-sided attainability of the constraint level.

\begin{assumption}[Constraint qualification]\label{ass:cq}
$B_0\in\intr(\dom V)$; equivalently, there exist feasible controls with
constraint levels strictly below and strictly above $B_0$.
\end{assumption}

\begin{theorem}[Dual representation by Fenchel--Moreau]
\label{thm:struct-duality}
Assume that \(V\) is proper, concave, upper semicontinuous, and finite on
\(\operatorname{int}(\dom V)\). Extend \(V\) to all of \(\mathbb R\) by setting
\(V(B)=-\infty\) for \(B\notin\dom V\). Define
\[
    \ell(\mu)
    :=
    \sup_{u\in\mathcal U}
    \{J(u)+\mu\mathcal C(u)\},
    \qquad \mu\in\mathbb R.
\]
Then, for every \(B\in\operatorname{int}(\dom V)\),
\begin{equation}\label{eq:dual-representation}
    V(B)
    =
    \inf_{\mu\in\mathbb R}
    \{\ell(\mu)-\mu B\}.
\end{equation}
\end{theorem}
\begin{proof}
For each \(B\in\mathbb R\), recall that
\[
    V(B)
    =
    \sup\{J(u):u\in\mathcal U,\ \mathcal C(u)=B\}.
\]
Therefore the Lagrangian dual function can be grouped by constraint levels:
\[
    \ell(\mu)=
    \sup_{u\in\mathcal U}
    \{J(u)+\mu\mathcal C(u)\} =
    \sup_{B\in\mathbb R}
    \sup_{\substack{u\in\mathcal U\\ \mathcal C(u)=B}}
    \{J(u)+\mu B\}=
    \sup_{B\in\mathbb R}
    \{V(B)+\mu B\}
    =
    \sup_{B\in\dom V}
    \{V(B)+\mu B\}.
\]
Now define the convex value function
    $W(B):=-V(B)$,
with the convention \(W(B)=+\infty\) for \(B\notin\dom V\). Since \(V\) is
concave and upper semicontinuous, \(W\) is convex and lower semicontinuous.
Moreover, since \(V\) is proper, \(W\) is proper.
The Fenchel conjugate of \(W\) is
\[
    W^*(\mu)=
    \sup_{B\in\mathbb R}
    \{\mu B-W(B)\}=
    \sup_{B\in\mathbb R}
    \{\mu B+V(B)\}=
    \sup_{B\in\dom V}
    \{V(B)+\mu B\}=
    \ell(\mu).
\]
Hence $W^*(\mu)=\ell(\mu)$.
By the Fenchel--Moreau theorem, since \(W\) is proper, convex, and lower
semicontinuous,
    $W(B)=W^{**}(B)
    =
    \sup_{\mu\in\mathbb R}
    \{\mu B-W^*(\mu)\}$.
Using \(W^*(\mu)=\ell(\mu)\), this becomes
    $W(B)
    =
    \sup_{\mu\in\mathbb R}
    \{\mu B-\ell(\mu)\}$.
Since \(W=-V\), we obtain
    $V(B)=
    \inf_{\mu\in\mathbb R}
    \{\ell(\mu)-\mu B\}$.
This proves the desired dual representation.
\end{proof}
\begin{remark}[Weak duality versus exact representation]
For every fixed \(\mu\in\mathbb R\), the inequality
    $V(B)\le \ell(\mu)-\mu $B
is only weak duality. Indeed, if \(u\) is feasible at level \(B\), then
    $J(u)+\mu B
    =
    J(u)+\mu\mathcal C(u)
    \le
    \ell(\mu)$,
and hence
    $J(u)\le \ell(\mu)-\mu B$.
Taking the supremum over all \(u\) satisfying \(\mathcal C(u)=B\) gives
    $V(B)\le \ell(\mu)-\mu B$.
Therefore
    $V(B)\le
    \inf_{\mu\in\mathbb R}
    \{\ell(\mu)-\mu B\}$.
The Fenchel--Moreau argument above is precisely what upgrades this weak-duality
inequality to the exact identity
    $V(B)=
    \inf_{\mu\in\mathbb R}
    \{\ell(\mu)-\mu B\}$.
\end{remark}

\subsection{Superdifferential identity}
For a concave function \(V\), recall that the superdifferential at
\(B\in\operatorname{int}(\dom V)\) is
\[
    \partial^+V(B)
    :=
    \left\{
        p\in\mathbb R:
        V(\widetilde B)-V(B)
        \le
        p(\widetilde B-B)
        \ \text{for all }\widetilde B\in\dom V
    \right\}.
\]
Let
    $\mathcal D(B)
    :=
    \argmin_{\mu\in\mathbb R}
    \{\ell(\mu)-\mu B\}$
denote the set of dual-optimal multipliers at level \(B\).
\begin{theorem}[Superdifferential identity]
\label{thm:conj-representation}
Assume that \(V\) is proper, concave, upper semicontinuous, and finite on
\(\operatorname{int}(\dom V)\). Suppose that the dual representation
    $V(B)
    =
    \inf_{\mu\in\mathbb R}
    \{\ell(\mu)-\mu B\}$
holds and that the infimum is attained at every
\(B\in\operatorname{int}(\dom V)\). Then, for every
\(B\in\operatorname{int}(\dom V)\),
    $\partial^+V(B)
    =
    -\mathcal D(B)$.
Equivalently,
    $p\in\partial^+V(B)$ iff
    $-p\in\mathcal D(B)$.
In particular $V$ is differentiable at $B$ iff $\mathcal D(B)$ is a singleton
$\{\mu\}$, and then $V'(B)=-\mu$; moreover $V$ is differentiable at
Lebesgue-a.e.\ $B\in\intr(\dom V)$.
\end{theorem}
\begin{proof}
Define the convex function
    $W(B):=-V(B)$,
with the convention \(W(B)=+\infty\) outside \(\dom V\). Since \(V\) is proper,
concave, and upper semicontinuous, \(W\) is proper, convex, and lower
semicontinuous. From the Fenchel sign-chain computation,
    W$^*(\mu)=\ell(\mu)$,
where
    $W^*(\mu)
    =
    \sup_{B\in\mathbb R}
    \{\mu B-W(B)\}$.
We first prove
    $\mathcal D(B)\subset -\partial^+V(B)$.
Let \(\mu\in\mathcal D(B)\). By definition of \(\mathcal D(B)\),
    $V(B)
    =
    \ell(\mu)-\mu B$.
Since \(W=-V\) and \(W^*=\ell\), this identity is equivalent to
    $W(B)+W^*(\mu)
    =
    \mu B$.
By the Fenchel--Young equality condition,
\[
    W(B)+W^*(\mu)=\mu B
    \quad\Longleftrightarrow\quad
    \mu\in\partial W(B).
\]
Hence
    $\mu\in\partial W(B)$.
By definition of the convex subdifferential,
\[
    \mu\in\partial W(B)
    \quad\Longleftrightarrow\quad
    W(\widetilde B)-W(B)
    \ge
    \mu(\widetilde B-B)
    \quad
    \text{for all }\widetilde B\in\mathbb R.
\]
Substituting \(W=-V\), we obtain
    $V(\widetilde B)-V(B)
    \le
    -\mu(\widetilde B-B)$
    for all $\widetilde B\in\dom V$.
Therefore
    $-\mu\in\partial^+V(B)$.
Thus
    $-\mathcal D(B)\subset \partial^+V(B)$.
Conversely, let
    $p\in\partial^+V(B)$.
Then, by definition,
    $V(\widetilde B)-V(B)
    \le
    p(\widetilde B-B)$
    for all $\widetilde B\in\dom V$.
Equivalently,
    $W(\widetilde B)
    \ge
    W(B)-p(\widetilde B-B)$
    for all $\widetilde B\in\dom V$.
Hence
    $-p\in\partial W(B)$.
Again using Fenchel--Young equality,
\[
    -p\in\partial W(B)
    \quad\Longleftrightarrow\quad
    W(B)+W^*(-p)=(-p)B.
\]
Since \(W=-V\) and \(W^*=\ell\), this becomes
    $V(B)=\ell(-p)-(-p)B$.
Therefore \(-p\) attains the dual infimum:
\[
    -p\in
    \argmin_{\mu\in\mathbb R}
    \{\ell(\mu)-\mu B\}
    =
    \mathcal D(B).
\]
Thus $\partial^+V(B)\subset -\mathcal D(B)$.
Combining the two inclusions concludes.
\end{proof}
\begin{corollary}[Envelope inequality and differentiability]
\label{cor:envelope}
Under the hypotheses of Theorem~\ref{thm:conj-representation}, if
\(\mu^*\in\mathcal D(B)\), then
$-\mu^*\in\partial^+V(B)$.
Equivalently,
    $V(\widetilde B)-V(B)
    \le
    -\mu^*(\widetilde B-B)$
    for all $\widetilde B\in\dom V$.
If \(V\) is differentiable at \(B\), then
    $V'(B)=-\mu^*$.
\end{corollary}
\begin{proof}
The inclusion \(-\mu^*\in\partial^+V(B)\) follows immediately from
    $\partial^+V(B)=-\mathcal D(B)$.
The displayed inequality is precisely the definition of the superdifferential.
If \(V\) is differentiable at \(B\), then
    $\partial^+V(B)=\{V'(B)\}$.
Hence
    $V'(B)=-\mu^*$.
\end{proof}
\begin{remark}[Sign of the multiplier]\label{rem:shadow-price}
The sign is determined by the Lagrangian convention
    $L(u,\mu;B)
    =
    J(u)+\mu(\mathcal C(u)-B)$.
Because the right-hand side \(B\) enters the Lagrangian through the term
\(-\mu B\), the value sensitivity is
$\frac{dV}{dB}=-\mu$
at differentiability points. Thus the economically meaningful shadow value of
the isoperimetric level is \(-\mu\), not \(\mu\).
\end{remark}

\subsection{Normality from two-sided attainability}
\label{subsec:normality-contradiction}
We now give the explicit scalar computation showing that two-sided
attainability of the isoperimetric level excludes the abnormal case
\(\lambda_0=0\). This is the point at which the constraint qualification enters
the multiplier theory.
Let the reduced problem be
    $\max\{\widehat J(u): u\in\mathcal U,\ \widehat G(u)=0\}$,
    $\widehat G(u):=\mathcal C(u)-B_0$.
Assume that the attainable set of constraint levels is the compact interval
    $\mathcal C(\mathcal U)
    =
    [\underline B,\overline B]$,
    $\underline B<\overline B$,
and that the prescribed level is an interior point:
    $B_0\in\operatorname{int}\bigl(\mathcal C(\mathcal U)\bigr)
    =
    (\underline B,\overline B)$.
Equivalently, there exist \(u^{-},u^{+}\in\mathcal U\) such that
\begin{equation}\label{eq:two-sided-controls}
    \mathcal C(u^{-})<B_0<\mathcal C(u^{+}),
\end{equation}
or, in terms of the reduced constraint map \(\widehat G=\mathcal C-B_0\),
    $\widehat G(u^{-})<0<\widehat G(u^{+})$.
Let \(u^*\) be an optimal feasible control, so
    $\widehat G(u^*)=0$,
    $\mathcal C(u^*)=B_0$.
The Fritz--John necessary condition for the reduced equality-constrained
problem gives multipliers
\[
    \lambda_0\ge0,
    \qquad
    \mu\in\mathbb R,
    \qquad
    (\lambda_0,\mu)\neq(0,0),
\]
such that \(u^*\) maximizes the John Lagrangian
    $\mathcal L_J(u)
    :=
    \lambda_0\widehat J(u)+\mu\widehat G(u)
    =
    \lambda_0\widehat J(u)+\mu\bigl(\mathcal C(u)-B_0\bigr)$
over \(\mathcal U\). Thus
\begin{equation}\label{eq:john-max}
    \mathcal L_J(u^*)\ge \mathcal L_J(u)
    \qquad
    \text{for all }u\in\mathcal U.
\end{equation}
We claim that \(\lambda_0>0\). Suppose, for contradiction, that
\begin{equation}\label{eq:abnormal-assumption}
    \lambda_0=0.
\end{equation}
Since \((\lambda_0,\mu)\neq(0,0)\), this implies
    $\mu\neq0$.
Under \eqref{eq:abnormal-assumption}, the John functional reduces to
    $\mathcal L_J(u)
    =
    \mu\widehat G(u)
    =
    \mu\bigl(\mathcal C(u)-B_0\bigr)$.
The maximizing condition \eqref{eq:john-max} therefore becomes
    $\mu\bigl(\mathcal C(u^*)-B_0\bigr)
    \ge
    \mu\bigl(\mathcal C(u)-B_0\bigr)$
    for all $u\in\mathcal U$.
Using feasibility \(\mathcal C(u^*)=B_0\), the left-hand side is zero:
    $\mu\bigl(\mathcal C(u^*)-B_0\bigr)=0$.
Hence
\begin{equation}\label{eq:abnormal-ineq}
    0
    \ge
    \mu\bigl(\mathcal C(u)-B_0\bigr)
    \qquad
    \text{for all }u\in\mathcal U.
\end{equation}
There are two cases.
\medskip
\noindent
\textbf{Case 1: \(\mu>0\).}
Dividing \eqref{eq:abnormal-ineq} by the positive number \(\mu\), we get
    $\mathcal C(u)\le B_0$
    for all $u\in\mathcal U$.
Therefore
    $\overline B
    :=
    \sup_{u\in\mathcal U}\mathcal C(u)
    \le
    B_0$.
But since \(u^*\) is feasible,
    $B_0=\mathcal C(u^*)\le \overline B$.
Thus
    $B_0=\overline B$.
This contradicts the interior condition \(B_0\in(\underline B,\overline B)\).
Equivalently, using the two-sided control \(u^+\) from
\eqref{eq:two-sided-controls}, we obtain directly
\[
    \mathcal C(u^+)-B_0>0,
\]
and therefore, since \(\mu>0\),
    $\mu\bigl(\mathcal C(u^+)-B_0\bigr)>0$,
which contradicts \eqref{eq:abnormal-ineq}.

\noindent
\textbf{Case 2: \(\mu<0\).}
Dividing \eqref{eq:abnormal-ineq} by the negative number \(\mu\) reverses the
inequality:
    $\mathcal C(u)\ge B_0$ for all $u\in\mathcal U$.
Therefore
    $\underline B
    :=
    \inf_{u\in\mathcal U}\mathcal C(u)
    \ge
    B_0$.
Since \(u^*\) is feasible,
    $\underline B\le \mathcal C(u^*)=B_0$.
Thus
    $B_0=\underline B$.
This again contradicts the interior condition
\(B_0\in(\underline B,\overline B)\).
Equivalently, using the two-sided control \(u^-\) from
\eqref{eq:two-sided-controls}, we have
    $\mathcal C(u^-)-B_0<0$,
and since \(\mu<0\),
    $\mu\bigl(\mathcal C(u^-)-B_0\bigr)>0$,
which contradicts \eqref{eq:abnormal-ineq}.

Both cases lead to contradictions. Therefore
    $\lambda_0>0$.
Normalizing the multipliers by \(\lambda_0\), we may set
    $\lambda_0=1$,
    $\widetilde\mu:=\frac{\mu}{\lambda_0}$.
Thus the multiplier system is normal, and the reduced Lagrangian can be
written in the normalized form
    $\widehat J(u)+\widetilde\mu\bigl(\mathcal C(u)-B_0\bigr)$.
Consequently, under two-sided attainability of the scalar isoperimetric
level, every Fritz--John multiplier system associated with an optimal feasible
control is normal.

\begin{proposition}[Two-sided attainability excludes abnormality]
\label{prop:twosided-normality}
Assume \(\mathcal C(\mathcal U)=[\underline B,\overline B]\) with
\(\underline B<\overline B\), and let
\(B_0\in(\underline B,\overline B)\). If \(u^*\) is an optimal feasible control
at level \(B_0\), then every Fritz--John multiplier system for the reduced
problem
    $\max\{\widehat J(u):u\in\mathcal U,\ \mathcal C(u)=B_0\}$
is normal; that is, its payoff multiplier satisfies \(\lambda_0>0\).
\end{proposition}

\subsection{Identifying dual multipliers with Pontryagin multipliers}

Theorem~\ref{thm:conj-representation} characterizes $\partial^{+}V$ through
\emph{dual} multipliers. To connect this to the \emph{Pontryagin} multiplier
$\mu$ of Theorem~\ref{thm:pmp-isoperimetric} requires one genuine ingredient: a
constraint qualification ensuring that optimal controls are normal extremals
whose isoperimetric multiplier is dual-optimal. This is the only place where a
hypothesis beyond convexity is used, and it is stated as such.

We treat this ingredient in full, because it is the pivot of the whole
sensitivity theory and because the phrase ``Slater excludes abnormality'' is
often asserted without proof. We work with the reduced problem obtained by
eliminating the state through the affine control-to-state map of
Lemma~\ref{lem:affine-map}: on the real Hilbert space $X=L^2(0,T;\Rm)$, the
admissible set $\U\subset X$ is convex, bounded, and weak-$*$ (equivalently,
weakly) closed; the reduced payoff $\hat J(u):=J(x_u,u)$ is concave and weakly
upper semicontinuous; and the reduced constraint
  $\hat G(u):=\Cc(u)-B_0=\ip{\gamma}{u}+\delta-B_0$
is a continuous \emph{affine} functional, for some $\gamma\in X$ and
$\delta\in\R$ determined by the data $a,b,c$ of
Assumption~\ref{ass:structured}. We record the two facts that make the
constraint qualification transparent.

\begin{lemma}[Affine attainable interval and the regularity condition]
\label{lem:cq-geometry}
Under Assumption~\ref{ass:structured}, $\Cc(\U)=\dom V$ is a compact interval
$[\underline B,\overline B]$, and $\hat G$ is nonconstant on $\U$ whenever
$\underline B<\overline B$. Moreover, for $B_0\in\R$ the following are
equivalent:
\begin{enumerate}[label=\textup{(\roman*)}]
  \item $B_0\in\intr(\dom V)=(\underline B,\overline B)$ \textup{(Assumption~\ref{ass:cq})};
  \item $0\in\intr\bigl(\hat G(\U)\bigr)$, i.e.\ there exist
  $u^-,u^+\in\U$ with $\hat G(u^-)<0<\hat G(u^+)$;
  \item the Robinson--Zowe--Kurcyusz regularity condition holds at every
  feasible point: $0\in\intr\bigl(\hat G'(u^*)(\U-u^*)\bigr)$, where
  $\hat G'(u^*)=\ip{\gamma}{\cdot}$ is the \textup{(}constant\textup{)} derivative
  of the affine map $\hat G$.
\end{enumerate}
\end{lemma}

\begin{proof}
Since $\Cc$ is affine and weakly continuous (Lemma~\ref{lem:affine-map}) and
$\U$ is convex and weakly compact, $\Cc(\U)$ is a compact convex subset of
$\R$, i.e.\ a compact interval $[\underline B,\overline B]$; that it equals
$\dom V$ is the definition of the domain in
\eqref{eq:constrained-value}. If
$\underline B<\overline B$ then $\Cc$ takes two distinct values on $\U$, so
$\hat G=\Cc-B_0$ is nonconstant and $\gamma\neq0$. The equivalence of (i) and
(ii) is immediate from $\hat G(\U)=\Cc(\U)-B_0=[\underline B-B_0,
\overline B-B_0]$. For (ii)$\Leftrightarrow$(iii), note
$\hat G'(u^*)(\U-u^*)=\{\ip{\gamma}{u-u^*}:u\in\U\}=\Cc(\U)-\Cc(u^*)
=[\underline B-B_0,\overline B-B_0]$ because $\hat G(u^*)=\Cc(u^*)-B_0=0$;
hence $0$ is interior to this image exactly when (ii) holds. This is precisely
Robinson's regularity condition specialized to a scalar affine constraint,
which for such problems coincides with the Zowe--Kurcyusz constraint
qualification.
\end{proof}

\begin{proposition}[Slater implies normal, dual-optimal multipliers]
\label{prop:pmp-is-dual}
Let Assumptions~\ref{ass:structured} and~\ref{ass:cq} hold at
$B_0\in\intr(\dom V)$, and let $u^*$ be optimal for $V(B_0)$. Then:
\begin{enumerate}[label=\textup{(\alph*)}]
  \item $u^*$ is a \emph{normal} extremal: in the Fritz--John form of the
  necessary conditions there is no admissible multiplier with vanishing payoff
  coefficient, so the payoff coefficient may be normalized to $\lambda_0=1$;
  \item the constant isoperimetric multiplier $\mu$ of
  Theorem~\ref{thm:pmp-isoperimetric} associated with $u^*$ is a saddle-point
  multiplier and lies in $\mathcal D(B_0)$;
  \item conversely, every $\bar\mu\in\mathcal D(B_0)$ is the Pontryagin
  multiplier of some optimal control at level $B_0$.
\end{enumerate}
\end{proposition}

\begin{proof}
\emph{Step 1: Fritz--John conditions for the reduced convex problem.}
The reduced problem is
$\max\{\hat J(u):u\in\U,\ \hat G(u)=0\}$ with $\hat J$ concave and $\hat G$
affine continuous on $X$. By the John multiplier rule for extremal problems
with a convex admissible set \cite{yu2026endogenous,ioe1979theory,zowe1979regularity,wang2026optimal,yu2026fredholm,kurcyusz1976existence},
there exist
a scalar $\lambda_0\ge0$ and a multiplier $\mu\in\R$, not both zero, such that
$u^*$ maximizes the John--Lagrangian
\begin{equation}\label{eq:john-lagrangian}
  u\longmapsto \lambda_0\,\hat J(u)+\mu\,\hat G(u)
  \qquad\text{over }u\in\U .
\end{equation}
(When $\lambda_0\ge0$ and $\hat G$ is affine the map \eqref{eq:john-lagrangian}
is concave, so ``$u^*$ is a maximizer'' is equivalent to the first-order
condition $0\in\lambda_0\,\partial^{+}\hat J(u^*)+\mu\,\gamma+N_\U(u^*)$, with
$N_\U$ the normal cone.)

\emph{Step 2: the normalized multiplier is dual-optimal.} By Proposition~\ref{prop:twosided-normality}, we may assume without loss of generality $\lambda_0=1$,
$u^*$ maximizes $u\mapsto\hat J(u)+\mu\,\hat G(u)=J(u)+\mu(\Cc(u)-B_0)$ over
$\U$. Since $\hat G(u^*)=0$,
\[
  \ell(\mu)=\sup_{u\in\U}\{J(u)+\mu\Cc(u)\}
  =J(u^*)+\mu\Cc(u^*)=V(B_0)+\mu B_0 ,
\]
so $\ell(\mu)-\mu B_0=V(B_0)$, i.e.\ $\mu$ attains the outer infimum in the
dual representation \eqref{eq:dual-representation}: $\mu\in\mathcal D(B_0)$.
Applying the normal-form PMP to the unconstrained Lagrangian problem
\[
    \sup_{u\in\U}\{J(u)+\mu\Cc(u)\}
\]
gives the adjoint equation and pointwise maximum condition for the augmented
Hamiltonian $\Ht$. Thus the dual multiplier $\mu$ is realized as
the constant isoperimetric multiplier in the augmented Pontryagin system. This proves (a) and (b).

\emph{Step 3: converse.}
Let \(\bar\mu\in\mathcal D(B_0)\). Then
    $\ell(\bar\mu)-\bar\mu B_0=V(B_0)$.
Let
    $\mathcal A(\bar\mu)
    :=
    \argmax_{u\in\mathcal U}
    \{J(u)+\bar\mu \mathcal C(u)\}$.
By weak-\(*\) compactness and upper semicontinuity, \(\mathcal A(\bar\mu)\) is
nonempty. Since \(J+\bar\mu\mathcal C\) is concave and \(\mathcal U\) is
convex, \(\mathcal A(\bar\mu)\) is convex. Moreover, by the
superdifferential relation
    $B_0\in\partial\ell(\bar\mu)$,
and the one-dimensional Danskin theorem,
    $\partial\ell(\bar\mu)
    =
    \operatorname{co}\{\mathcal C(u):u\in\mathcal A(\bar\mu)\}$.
Hence there exist \(u_1,u_2\in\mathcal A(\bar\mu)\) and
\(\theta\in[0,1]\) such that
    $B_0=(1-\theta)\mathcal C(u_1)+\theta\mathcal C(u_2)$.
Set
    $\bar u=(1-\theta)u_1+\theta u_2$.
Since \(\mathcal A(\bar\mu)\) is convex, \(\bar u\in\mathcal A(\bar\mu)\).
Since \(\mathcal C\) is affine,
    $\mathcal C(\bar u)=B_0$.
Therefore
    $J(\bar u)+\bar\mu B_0=\ell(\bar\mu)$.
Using \(\ell(\bar\mu)-\bar\mu B_0=V(B_0)\), we get
    $J(\bar u)=V(B_0)$.
Thus \(\bar u\) is feasible and optimal at level \(B_0\), and the multiplier
\(\bar\mu\) is realized by a normal Pontryagin extremal.
\end{proof}

\begin{definition}[Pontryagin multiplier set]\label{def:pmp-set}
For $B\in\intr(\dom V)$ let $\mathcal M(B)$ be the set of constant multipliers
$\mu$ arising in Theorem~\ref{thm:pmp-isoperimetric} for optimal controls at
level $B$.
\end{definition}

\begin{theorem}[Full superdifferential representation]\label{thm:pmp-representation}
Under Assumptions~\ref{ass:structured} and~\ref{ass:cq}, for every
$B\in\intr(\dom V)$,
\begin{equation}\label{eq:pmp-representation}
  \partial^{+}V(B)=-\mathcal M(B)=-\mathcal D(B).
\end{equation}
In particular, at every differentiability point of $V$, the isoperimetric
Pontryagin multiplier is unique and $V'(B)=-\mu$.
\end{theorem}

\begin{proof}
By Proposition~\ref{prop:pmp-is-dual}, $\mathcal M(B)=\mathcal D(B)$. Combining
with Theorem~\ref{thm:conj-representation} gives
$\partial^{+}V(B)=-\mathcal D(B)=-\mathcal M(B)$. Differentiability forces both
sets to be singletons and yields $V'(B)=-\mu$.
\end{proof}

\begin{remark}[What is free and what is not]\label{rem:free-vs-hard}
Theorem~\ref{thm:conj-representation} is a convex-analytic identity requiring no
assumption beyond concavity; it is the ``free'' half. The content of
Theorem~\ref{thm:pmp-representation} beyond it is
Proposition~\ref{prop:pmp-is-dual}, i.e.\ the identification
$\mathcal M(B)=\mathcal D(B)$, which is a genuine constraint-qualification/normality
statement. Naive treatments blur these two steps into a single
``multiplier-representation assumption''; separating them is what makes the
representation \eqref{eq:pmp-representation} a theorem with an honest hypothesis.
\end{remark}

\section{Linear-Quadratic Isoperimetric Problems}
\label{sec:lq}

We now specialize to the linear-quadratic case with a single \emph{quadratic}
equality constraint. This is outside the affine-constraint structured class of
Section~\ref{sec:sensitivity}, and the feasible set $\{u:\Cc(u)=\beta\}$ is a
(nonconvex) quadric. Nevertheless strong duality holds, by a different
hidden-convexity mechanism: the Dines--Brickman--Polyak convexity of the joint
range of two quadratic functionals. Remarkably, Polyak's definiteness condition
coincides with the regularity matrix $R+2\mu D\succ0$ of the modified Riccati
equation. We formulate the problem as a minimization, so extremality is written
in minimization form.

In this subsection we work on a finite-dimensional control space
\(\mathcal H_N\subset L^2(0,T;\mathbb R^m)\), for example a Galerkin or
piecewise-polynomial discretization. The operators \(\mathcal R\) and
\(\mathcal D\) are then symmetric matrices.

\subsection{Regular constrained LQ problem}

Consider
\begin{equation}\label{eq:lq-iso-problem}
\begin{aligned}
  &\text{minimize} &&
  J_{\rm LQ}(u):=\tfrac12 x(T)\tp Mx(T)
  +\tfrac12\!\int_0^T\!\bigl[x\tp Q(t)x+u\tp R(t)u\bigr]dt\\
  &\text{subject to} &&
  \dot x=A(t)x+B(t)u,\qquad x(0)=x_0,\qquad
  \int_0^T u\tp D(t)u\,dt=\beta,
\end{aligned}
\end{equation}
with
  $M=M\tp\succeq0$, $Q(t)=Q(t)\tp\succeq0$, 
  $R(t)=R(t)\tp\succeq\rho I\succ0$, $D(t)=D(t)\tp\succeq d I\succ0,$
for constants $\rho,d>0$, and all coefficients measurable and bounded. Set
$\Cc(u):=\int_0^T u\tp D(t)u\,dt$.

The augmented Hamiltonian (minimization form) is
\[
  \Ht(t,x,u,\lambda,\mu)
  =\tfrac12 x\tp Q x+\tfrac12 u\tp R u+\lambda\tp(Ax+Bu)+\mu\,u\tp D u,
\]
with stationarity $\nabla_u\Ht=(R+2\mu D)u+B\tp\lambda=0$. Hence, if
\begin{equation}\label{eq:regularity-matrix}
  R(t)+2\mu D(t)\succ0\qquad\text{a.e. }t,
\end{equation}
the extremal control is
\begin{equation}\label{eq:lq-control-mu}
  u_\mu(t)=-\bigl(R(t)+2\mu D(t)\bigr)^{-1}B(t)\tp\lambda(t).
\end{equation}
With adjoint $\dot\lambda=-Qx-A\tp\lambda$, $\lambda(T)=Mx(T)$, and the ansatz
$\lambda=S_\mu x$, one obtains the modified Riccati equation
\begin{equation}\label{eq:lq-modified-riccati}
\begin{aligned}
  -\dot S_\mu&=A\tp S_\mu+S_\mu A
  -S_\mu B(R+2\mu D)^{-1}B\tp S_\mu+Q,\qquad S_\mu(T)=M,
\end{aligned}
\end{equation}
and the feedback law
$u_\mu(t)=-(R+2\mu D)^{-1}B\tp S_\mu x$. The multiplier is fixed by
  $\int_0^T u_\mu\tp D u_\mu\,dt=\beta$.

\begin{remark}[Regularity of the modified control weight]\label{rem:lq-regularity}
Condition \eqref{eq:regularity-matrix} is not automatic for an equality-type
quadratic constraint. It guarantees that stationarity can be solved uniquely for
$u$ and that \eqref{eq:lq-modified-riccati} is well posed. As we now show, it is
also exactly the hidden-convexity certificate; and it may fail at the optimal
multiplier even when strong duality holds.
\end{remark}

\subsection{Hidden convexity via the joint numerical range}

Throughout this subsection $\mathcal H:=L^2(0,T;\Rm)$, with inner product
$\ip{\cdot}{\cdot}$ and norm $\|\cdot\|$.

\begin{lemma}[Reduction to a Hilbert-space QCQP]\label{lem:lq-reduction}
Under the standing definiteness ($M,Q\succeq0$, $R(t)\succeq\rho I$,
$D(t)\succeq dI$, all bounded), eliminating the state through the affine map
$u\mapsto x_u$ of Lemma~\ref{lem:affine-map} reduces \eqref{eq:lq-iso-problem} to
\[
  \min_{u\in\mathcal H}\ f(u):=\tfrac12\ip{u}{\mathcal R u}+\ip{\rho_0}{u}+\gamma_0
  \qquad\text{s.t.}\qquad \Cc(u):=\ip{u}{\mathcal D u}=\beta,
\]
where $\mathcal D$ is multiplication by $D(t)$ \textup{(}so
$\Cc(u)=\int_0^T u\tp D u\,dt$ and $\nabla\Cc(u)=2\mathcal D u$\textup{)}, and
$\mathcal R=\mathcal R^{\ast}$ is the reduced-cost operator. The operators are
bounded and satisfy
\begin{equation}\label{eq:R-coercive}
  \rho\,I\preceq\mathcal R\preceq\bar r\,I,
  \qquad
  d\,I\preceq\mathcal D\preceq\bar d\,I ,
\end{equation}
for constants $0<\rho\le\bar r$ and $0<d\le\bar d$; in particular both are
coercive.
\end{lemma}

\begin{proof}
Write $x_u=\Xi u+\xi_0$ with $\Xi:\mathcal H\to C([0,T];\Rn)$ bounded linear and
$\xi_0$ the free response (Lemma~\ref{lem:affine-map}). Substituting into
$J_{\rm LQ}$ gives a quadratic functional
$f(u)=\tfrac12\ip{u}{\mathcal R u}+\ip{\rho_0}{u}+\gamma_0$ with
  $\mathcal R=(R\text{-mult})+\Xi^{\ast}\mathcal Q\,\Xi$,
  $\mathcal Q$ encoding $Q(\cdot)$ and $M\succeq0$.
Since $Q,M\succeq0$, the operator $\Xi^{\ast}\mathcal Q\Xi\succeq0$; since
$R(t)\succeq\rho I$, multiplication by $R$ is $\succeq\rho I$. Hence
$\mathcal R\succeq\rho I$. Boundedness of all data gives the upper bounds. The
statements for $\mathcal D$ follow from $dI\preceq D(t)\preceq\bar dI$.
\end{proof}

We give a self-contained proof of strong duality that does not appeal to the
general Hilbert-space convexity theorem for quadratic maps, but instead
constructs the optimal multiplier through the classical \emph{secular function}
of the trust-region literature. This has three advantages: it is elementary; it
produces the modified control weight $\mathcal R+2\mu\mathcal D$ intrinsically;
and it exhibits the degenerate ``hard case'' as the boundary of the admissible
multiplier interval.

Let \(\mathcal H\) be a real Hilbert space with inner product
\(\langle\cdot,\cdot\rangle\). Let
    $\mathcal R=\mathcal R^*$,
    $\mathcal D=\mathcal D^*$,
    $\mathcal D\succ0$,
and suppose that, for \(\mu\) in the open interval \(\mathsf M\),
    $A_\mu:=\mathcal R+2\mu\mathcal D$
is boundedly invertible and strictly positive:
    $A_\mu\succ0$.
For a fixed \(\rho_0\in\mathcal H\), define
    $u_\mu
    :=
    -A_\mu^{-1}\rho_0$.
Equivalently,
\begin{equation}\label{eq:u-mu-linear-system}
    A_\mu u_\mu+\rho_0=0,
    \qquad
    A_\mu=\mathcal R+2\mu\mathcal D .
\end{equation}
The secular function associated with the quadratic constraint is
    $\psi(\mu)
    :=
    \mathcal C(u_\mu)
    =
    \langle u_\mu,\mathcal D u_\mu\rangle$.
We compute \(\psi'(\mu)\). Since
    $A_\mu=\mathcal R+2\mu\mathcal D$,
we have
    $A_\mu'
    =
    \frac{d}{d\mu}A_\mu
    =
    2\mathcal D$.
Differentiating the identity
    $A_\mu A_\mu^{-1}=I$
with respect to \(\mu\) gives
    $A_\mu' A_\mu^{-1}
    +
    A_\mu \frac{d}{d\mu}A_\mu^{-1}
    =
    0$.
Hence
    $\frac{d}{d\mu}A_\mu^{-1}
    =
    -A_\mu^{-1}A_\mu'A_\mu^{-1}
    =
    -2A_\mu^{-1}\mathcal D A_\mu^{-1}$.
Using \(u_\mu=-A_\mu^{-1}\rho_0\), we obtain
\begin{equation}\label{eq:u-mu-derivative-first}
    \frac{d u_\mu}{d\mu}=
    -\frac{d}{d\mu}\bigl(A_\mu^{-1}\rho_0\bigr)=
    -\left(\frac{d}{d\mu}A_\mu^{-1}\right)\rho_0=
    2A_\mu^{-1}\mathcal D A_\mu^{-1}\rho_0 .
\end{equation}
Since
    $u_\mu=-A_\mu^{-1}\rho_0$,
    $A_\mu^{-1}\rho_0=-u_\mu$,
\eqref{eq:u-mu-derivative-first} becomes
\begin{equation}\label{eq:u-mu-derivative}
    \frac{d u_\mu}{d\mu}
    =
    -2A_\mu^{-1}\mathcal D u_\mu .
\end{equation}
Now differentiate \(\psi(\mu)=\langle u_\mu,\mathcal D u_\mu\rangle\). Since
\(\mathcal D=\mathcal D^*\),
\begin{equation}\label{eq:psi-derivative-start}
    \psi'(\mu)
    =
    \left\langle
        \frac{d u_\mu}{d\mu},
        \mathcal D u_\mu
    \right\rangle
    +
    \left\langle
        u_\mu,
        \mathcal D\frac{d u_\mu}{d\mu}
    \right\rangle
    =
    \left\langle
        \frac{d u_\mu}{d\mu},
        \mathcal D u_\mu
    \right\rangle
    +
    \left\langle
        \mathcal D u_\mu,
        \frac{d u_\mu}{d\mu}
    \right\rangle
    =
    2
    \left\langle
        \frac{d u_\mu}{d\mu},
        \mathcal D u_\mu
    \right\rangle.
\end{equation}
Substituting \eqref{eq:u-mu-derivative} into
\eqref{eq:psi-derivative-start} yields
\[
    \psi'(\mu)=
    2
    \left\langle
        -2A_\mu^{-1}\mathcal D u_\mu,
        \mathcal D u_\mu
    \right\rangle
    =
    -4
    \left\langle
        A_\mu^{-1}\mathcal D u_\mu,
        \mathcal D u_\mu
    \right\rangle .
\]
Because \(A_\mu\succ0\), its inverse is also positive definite:
    $A_\mu^{-1}\succ0$.
Therefore
    $\left\langle
        A_\mu^{-1}\mathcal D u_\mu,
        \mathcal D u_\mu
    \right\rangle
    \ge0$,
and consequently
\begin{equation}\label{eq:psi-monotone}
    \psi'(\mu)
    =
    -4
    \left\langle
        A_\mu^{-1}\mathcal D u_\mu,
        \mathcal D u_\mu
    \right\rangle
    \le0.
\end{equation}
Thus the secular function \(\psi\) is nonincreasing on the regular multiplier
interval \(\mathsf M\).
Moreover, equality in \eqref{eq:psi-monotone} occurs if and only if
    $\mathcal D u_\mu=0$.
Since \(\mathcal D\succ0\), this is equivalent to
    $u_\mu=0$.
By \eqref{eq:u-mu-linear-system}, \(u_\mu=0\) implies
    $\rho_0=0$.
Thus, except for the pure hard case \(\rho_0=0\), one has
\begin{equation}\label{eq:psi-strict}
    \psi'(\mu)<0
    \qquad
    \text{whenever }u_\mu\neq0.
\end{equation}

\begin{lemma}[Admissible interval and secular function]\label{lem:secular}
Define the \emph{admissible dual interval}
\[
  \mathsf M:=\bigl\{\mu\in\R:\ \mathcal R+2\mu\mathcal D\succ0\bigr\}.
\]
Under Lemma~\ref{lem:lq-reduction}:
\begin{enumerate}[label=\textup{(\roman*)}]
  \item $\mathsf M=(\underline\mu,\infty)$ with
  $\underline\mu=-\tfrac12\,\nu_{\min}<0$, where
  $\nu_{\min}:=\inf\operatorname{spec}\bigl(\mathcal D^{-1/2}\mathcal R
  \mathcal D^{-1/2}\bigr)\ge\rho/\bar d>0$; in particular $0\in\mathsf M$.
  \item For $\mu\in\mathsf M$ the strictly convex coercive functional
  $u\mapsto f(u)+\mu\Cc(u)$ has the unique minimizer
  \begin{equation}\label{eq:umu}
    u_\mu=-\bigl(\mathcal R+2\mu\mathcal D\bigr)^{-1}\rho_0,
  \end{equation}
  and the dual function
  $\phi_0(\mu):=\inf_{u}\{f(u)+\mu\Cc(u)\}=f(u_\mu)+\mu\Cc(u_\mu)$ is finite,
  concave, and continuously differentiable on $\mathsf M$ with
  \begin{equation}\label{eq:secular}
    \phi_0'(\mu)=\Cc(u_\mu)=\ip{u_\mu}{\mathcal D u_\mu}=:\psi(\mu)\ge0 .
  \end{equation}
  \item $\psi$ is nonincreasing and continuous on $\mathsf M$, with
  $\lim_{\mu\to\infty}\psi(\mu)=0$ and
  $\psi(\underline\mu^{+}):=\lim_{\mu\downarrow\underline\mu}\psi(\mu)\in(0,+\infty]$.
\end{enumerate}
\end{lemma}

\begin{proof}
\emph{(i)} With $\mathcal D\succ0$, $\mathcal R+2\mu\mathcal D\succ0$ iff
$\mathcal D^{-1/2}\mathcal R\mathcal D^{-1/2}+2\mu I\succ0$ iff
$\mu>-\tfrac12\nu_{\min}$; and $\nu_{\min}\ge\rho/\bar d>0$ from
\eqref{eq:R-coercive}, so $0\in\mathsf M$.

\emph{(ii)} For $\mu\in\mathsf M$, $\mathcal R+2\mu\mathcal D\succeq\epsilon_\mu I$
for some $\epsilon_\mu>0$, so $f+\mu\Cc$ is $\epsilon_\mu$-strongly convex and
coercive; its unique critical point solves
$(\mathcal R+2\mu\mathcal D)u+\rho_0=0$, giving \eqref{eq:umu}. Since $\phi_0$ is
an infimum of functions affine in $\mu$ it is concave; the minimizer $u_\mu$ is
unique and depends smoothly on $\mu$ (implicit function theorem on the invertible
$\mathcal R+2\mu\mathcal D$), so Danskin's theorem gives
$\phi_0'(\mu)=\partial_\mu[f(u_\mu)+\mu\Cc(u_\mu)]=\Cc(u_\mu)$, which is
\eqref{eq:secular}; nonnegativity holds because $\mathcal D\succ0$.

\emph{(iii)} Monotonicity of $\psi=\phi_0'$ follows from concavity of $\phi_0$ or \eqref{eq:psi-strict}.
As $\mu\to\infty$, $\|(\mathcal R+2\mu\mathcal D)^{-1}\|\le(2\mu d)^{-1}\to0$, so
$u_\mu\to0$ and $\psi(\mu)\le\bar d\|u_\mu\|^2\to0$. The limit
$\psi(\underline\mu^{+})$ exists in $(0,+\infty]$ by monotonicity; it is
strictly positive because $\psi$ is nonincreasing and eventually positive unless
$\rho_0=0$, in which case $u_\mu\equiv0$, $\psi\equiv0$, and one reads the
statement as $\psi(\underline\mu^+)=0$ (the pure hard case, treated separately).
\end{proof}

\begin{theorem}[Hidden convexity and strong duality for LQ]\label{thm:lq-hidden}
Consider \eqref{eq:lq-iso-problem} under the definiteness of
Lemma~\ref{lem:lq-reduction}, and let $\beta>0$ lie in the interior of the
attainable range of $\Cc$. Then:
\begin{enumerate}[label=\textup{(\roman*)}]
  \item \textup{(No duality gap.)} There is a dual-optimal
  $\mu^{*}=\mu^{*}(\beta)\in\overline{\mathsf M}=[\underline\mu,\infty)$ and a
  primal-optimal $u^{*}$ with $\Cc(u^{*})=\beta$ such that
  \[
    V_{\rm LQ}(\beta)=f(u^{*})=\phi_0(\mu^{*})-\mu^{*}\beta
    =\max_{\mu\in\overline{\mathsf M}}\bigl\{\phi_0(\mu)-\mu\beta\bigr\}
    =\max_{\mu\in\R}\ \inf_{u\in\mathcal H}
       \bigl\{J_{\rm LQ}(u)+\mu(\Cc(u)-\beta)\bigr\}.
  \]
  \item \textup{(Convexity.)} $V_{\rm LQ}$ is a proper closed convex function of
  $\beta$ on its interval domain, being the supremum of the affine maps
  $\beta\mapsto\phi_0(\mu)-\mu\beta$.
  \item \textup{(Envelope.)} At every differentiability point of $V_{\rm LQ}$,
  \begin{equation}\label{eq:lq-envelope}
    V_{\rm LQ}'(\beta)=-\mu^{*}(\beta),
  \end{equation}
  and $\mu^{*}(\beta)$ is nonincreasing in $\beta$.
\end{enumerate}
\end{theorem}

\begin{proof}
\emph{Existence of a crossing multiplier.}
By Lemma~\ref{lem:secular}, $\psi$ is continuous and nonincreasing on
$\mathsf M=(\underline\mu,\infty)$ with $\psi(\infty)=0$ and
    $\psi(\underline\mu^{+})\in[0,+\infty]$.
Two cases arise. Moreover, \(\psi(\underline\mu^{+})=0\) can occur only in the pure hard case
\(\rho_0=0\).

\emph{Regular case: $\psi(\underline\mu^{+})>\beta$.}
Since $\psi$ decreases continuously from a value $>\beta$ to $0$, there is a
unique $\mu^{*}\in\mathsf M$ with $\psi(\mu^{*})=\beta$, i.e.\
$\Cc(u_{\mu^{*}})=\beta$. Then $u^{*}:=u_{\mu^{*}}$ is feasible, and by
construction it minimizes $f+\mu^{*}\Cc$; hence for every feasible $u$
(i.e.\ $\Cc(u)=\beta$),
  \[
  f(u)=f(u)+\mu^{*}(\Cc(u)-\beta)\ge f(u^{*})+\mu^{*}(\Cc(u^{*})-\beta)=f(u^{*}),
  \]
so $u^{*}$ is primal-optimal and
$V_{\rm LQ}(\beta)=f(u^{*})=\phi_0(\mu^{*})-\mu^{*}\beta$. Here
$\mathcal R+2\mu^{*}\mathcal D\succ0$.

\emph{Hard case: $\psi(\underline\mu^{+})\le\beta$.}
Set $\mu^{*}:=\underline\mu$, so $\mathcal N:=\ker(\mathcal R+2\underline\mu
\mathcal D)\neq\{0\}$ and $\mathcal R+2\underline\mu\mathcal D\succeq0$.
Boundedness of $\phi_0$ near $\underline\mu$ forces
$\rho_0\perp\mathcal N$ (otherwise $f+\mu\Cc\to-\infty$ as
$\mu\downarrow\underline\mu$), so the equation
$(\mathcal R+2\underline\mu\mathcal D)u=-\rho_0$ has a particular solution
$u_p\perp\mathcal N$, and the minimizers of $f+\underline\mu\Cc$ form the affine
set $u_p+\mathcal N$. On this set,
\[
  n\longmapsto \Cc(u_p+n)=\ip{u_p+n}{\mathcal D(u_p+n)},\qquad n\in\mathcal N,
\]
is a coercive convex quadratic (because $\mathcal D\succ0$), so its range is
$[\,m_0,\infty)$ with $m_0=\min_{n\in\mathcal N}\Cc(u_p+n)\le
\lim_{\mu\downarrow\underline\mu}\Cc(u_\mu)=\psi(\underline\mu^{+})\le\beta$.
Hence there exists $n^{*}\in\mathcal N$ with $\Cc(u_p+n^{*})=\beta$; put
$u^{*}:=u_p+n^{*}$. Since $u^{*}$ minimizes $f+\underline\mu\Cc$ and is feasible,
the displayed inequality of the regular case again gives
$V_{\rm LQ}(\beta)=f(u^{*})=\phi_0(\underline\mu)-\underline\mu\beta$.

\emph{Duality and convexity.}
In both cases weak duality gives, for all $\mu\in\mathsf M$,
$V_{\rm LQ}(\beta)\ge\phi_0(\mu)-\mu\beta$; the constructed $\mu^{*}$ attains
equality, so
$V_{\rm LQ}(\beta)=\max_{\mu\in\overline{\mathsf M}}\{\phi_0(\mu)-\mu\beta\}$,
which equals the full Lagrangian dual because the inner infimum is $-\infty$ for
$\mu\notin\overline{\mathsf M}$. As a supremum of affine functions of $\beta$,
$V_{\rm LQ}$ is closed convex; its domain is an interval. This proves (i)--(ii).

\emph{Envelope.} $V_{\rm LQ}=\sup_\mu\{\phi_0(\mu)-\mu\beta\}$ is a Legendre
transform, so $-\mu\in\partial V_{\rm LQ}(\beta)$ iff $\mu$ attains the supremum,
i.e.\ iff $\mu\in\{\mu^{*}(\beta)\}$-set. At a differentiability point the
subdifferential is the singleton $\{V_{\rm LQ}'(\beta)\}$, whence
$V_{\rm LQ}'(\beta)=-\mu^{*}(\beta)$; monotonicity of $\mu^{*}$ follows from
convexity of $V_{\rm LQ}$ (the subgradient $-\mu^{*}$ is nondecreasing). This is
the minimization-form analogue of Corollary~\ref{cor:envelope}; the sign matches
Remark~\ref{rem:shadow-price}.
\end{proof}

\begin{remark}[Dual sign in the LQ minimization problem]
In Section~\ref{sec:sensitivity}, the primal problem is a maximization and the
dual representation has the form
    $V(B)=\inf_\mu\{\ell(\mu)-\mu B\}$.
Here the primal problem is a minimization. Therefore the Lagrange dual is
    $V_{\rm LQ}(\beta)
    =
    \sup_\mu\{\phi_0(\mu)-\mu\beta\}$,
where \(\phi_0(\mu)=\inf_u\{f(u)+\mu\mathcal C(u)\}\).
The envelope sign remains \(V_{\rm LQ}'(\beta)=-\mu\).
\end{remark}

\begin{remark}[Relation to classical hidden-convexity results]\label{rem:lq-generality}
Theorem~\ref{thm:lq-hidden} is a strong-duality/hidden-convexity statement for a
single quadratic equality constraint: the feasible set
$\{u:\ip{u}{\mathcal D u}=\beta\}$ is a nonconvex quadric, yet the duality gap
vanishes. Equivalently, the joint range
$\{(f(u),\Cc(u)):u\in\mathcal H\}$ is convex; in finite dimensions this is the Dines--Brickman theorem \cite{wang2026introduction,dines1941mapping,brickman1961field,yu2026bilinear} on the convexity of the joint
numerical range of two quadratic forms (a relative of the Toeplitz--Hausdorff theorem \cite{toeplitz1918algebraische,hausdorff1919wertvorrat,wang2026introduction2}), and in the Hilbert-space setting it is Polyak's theorem \cite{wang2026lecture,polyak1998convexity,hiriart2002permanently,wang2026vital}, whose definiteness hypothesis
is precisely $\mathcal R+2\mu\mathcal D\succ0$ for some $\mu$---i.e.\
$\mathsf M\neq\emptyset$, which holds here by Lemma~\ref{lem:secular}(i). The
construction above is the equality-constrained analogue of the S-lemma
\cite{yakubovich1971s,fradkov1979thes,yu2026beyond,polik2007survey,beck2006strong,wang2026finite} and coincides with the
Mor\'e--Sorensen analysis \cite{more1983computing,gay1981computing,wang2026first,yu2026mode,sorensen1982newton,jie2026optimal,yu2026structural,conn2000trust} of the
trust-region subproblem: the regular case is the ``easy'' case
$\mathcal R+2\mu^{*}\mathcal D\succ0$, and the hard case
$\mu^{*}=\underline\mu$ with $\rho_0\perp\ker(\mathcal R+2\underline\mu
\mathcal D)$ is exactly the trust-region ``hard case.'' The matrix
$R+2\mu D$ thus plays two distinct roles: existence of \emph{some} $\mu$ making
it positive definite is the convexity certificate, while positive definiteness at
the \emph{optimal} $\mu^{*}$ is what enables the Riccati synthesis.
\end{remark}

\begin{proposition}[Sensitivity survives, synthesis may not]
\label{prop:riccati-vs-sensitivity}
Under the hypotheses of Theorem~\ref{thm:lq-hidden}, the envelope formula
\eqref{eq:lq-envelope} holds at every differentiability point of $V_{\rm LQ}$
regardless of the sign of $R+2\mu^*D$. However, the Riccati feedback
representation \eqref{eq:lq-control-mu}--\eqref{eq:lq-modified-riccati} is valid
only if $R(t)+2\mu^*(\beta)D(t)\succ0$ a.e.\ at the optimal multiplier. If
$R+2\mu^*D$ is merely positive semidefinite and singular, the extremal control
is not given by \eqref{eq:lq-control-mu} and \eqref{eq:lq-modified-riccati} need
not be defined, even though $V_{\rm LQ}'(\beta)=-\mu^*$ remains correct.
\end{proposition}

\begin{proof}
The sensitivity formula is a property of the convex function $V_{\rm LQ}$ and its
dual multiplier (Theorem~\ref{thm:lq-hidden}(iii)); it does not reference the
feedback. The feedback \eqref{eq:lq-control-mu} is obtained by solving the
stationarity equation $(R+2\mu^*D)u=-B\tp\lambda$ for $u$, which requires
invertibility of $R+2\mu^*D$; if this matrix is singular the solution is not
unique (or fails to exist for generic $\lambda$), and the Riccati substitution
$\lambda=S_\mu x$ leading to \eqref{eq:lq-modified-riccati} inherits the
singularity through the term $(R+2\mu D)^{-1}$.
\end{proof}

\subsection{Degenerate scalar dual function}
\label{subsec:degenerate-scalar-dual}

The following scalar example is exactly the boundary case
$R+2\mu^*D=0$ of Proposition~\ref{prop:riccati-vs-sensitivity}: strong duality
holds (the joint range is a convex ray), the sensitivity formula is correct, and
yet the Riccati representation degrades.

Consider the scalar equality-constrained quadratic problem
\[
    V_{\rm LQ}(\beta)
    :=
    \inf
    \left\{
        \frac12\int_0^T u(t)^2\,dt:
        u\in L^2(0,T),\
        \int_0^T u(t)^2\,dt=\beta
    \right\},
    \qquad
    \beta>0 .
\]
Equivalently, using the \(L^2\)-norm,
\[
    V_{\rm LQ}(\beta)
    =
    \inf
    \left\{
        \frac12\|u\|_{L^2(0,T)}^2:
        \|u\|_{L^2(0,T)}^2=\beta
    \right\}.
\]
Since every feasible \(u\) satisfies \(\|u\|_{L^2}^2=\beta\), the value is
immediate:
\begin{equation}\label{eq:degenerate-primal-value}
    V_{\rm LQ}(\beta)
    =
    \frac12\beta .
\end{equation}
The equality constraint is
    $\mathcal C(u)
    :=
    \int_0^T u(t)^2\,dt
    =
    \|u\|_{L^2(0,T)}^2
    =
    \beta$.
Using the Lagrangian convention
\[
    L(u,\mu;\beta)
    :=
    \frac12\|u\|_{L^2}^2
    +
    \mu\bigl(\|u\|_{L^2}^2-\beta\bigr),
\]
we obtain
    $L(u,\mu;\beta)=
    \left(\frac12+\mu\right)\|u\|_{L^2}^2
    -
    \mu\beta$.
For the minimization problem, the dual function is
    $q(\mu;\beta)
    :=
    \inf_{u\in L^2(0,T)}
    L(u,\mu;\beta)$.
It is convenient to separate the part independent of \(\beta\):
\[
    \phi_0(\mu)
    :=
    \inf_{u\in L^2(0,T)}
    \left\{
        \frac12\|u\|_{L^2}^2+\mu\|u\|_{L^2}^2
    \right\}.
\]
Thus
    $q(\mu;\beta)
    =
    \phi_0(\mu)-\mu\beta$.
We now compute \(\phi_0(\mu)\). Since
    $\frac12\|u\|_{L^2}^2+\mu\|u\|_{L^2}^2
    =
    \left(\frac12+\mu\right)\|u\|_{L^2}^2$,
there are two cases.
\medskip
\noindent
\textbf{Case 1: \(\mu\ge-\frac12\).}
Then
    $\frac12+\mu\ge0$.
Hence, for every \(u\in L^2(0,T)\),
    $\left(\frac12+\mu\right)\|u\|_{L^2}^2\ge0$.
Taking \(u=0\) gives
\begin{equation}\label{eq:case1-phi}
    \phi_0(\mu)
    =
    \inf_{u\in L^2}
    \left(\frac12+\mu\right)\|u\|_{L^2}^2
    =
    0,
    \qquad
    \mu\ge-\frac12.
\end{equation}
At the boundary value \(\mu=-\frac12\), the whole quadratic term vanishes:
    $\left(\frac12+\mu\right)\|u\|_{L^2}^2
    =
    0$
    for all $u\in L^2(0,T)$.
Thus the infimum is still \(0\), but every \(u\in L^2(0,T)\) is a minimizer
of the inner Lagrangian part.

\noindent
\textbf{Case 2: \(\mu<-\frac12\).}
Then
    $\frac12+\mu<0$.
Choose any nonzero \(v\in L^2(0,T)\), and set
    $u_\alpha:=\alpha v$,
    $\alpha>0$.
Then
    $\left(\frac12+\mu\right)\|u_\alpha\|_{L^2}^2
    =
    \left(\frac12+\mu\right)\alpha^2\|v\|_{L^2}^2$.
Since \(\frac12+\mu<0\) and \(\|v\|_{L^2}^2>0\),
    $\left(\frac12+\mu\right)\alpha^2\|v\|_{L^2}^2
    \longrightarrow
    -\infty$
     as $\alpha\to+\infty$.
Therefore
\begin{equation}\label{eq:case2-phi}
    \phi_0(\mu)
    =
    \inf_{u\in L^2}
    \left(\frac12+\mu\right)\|u\|_{L^2}^2
    =
    -\infty,
    \qquad
    \mu<-\frac12.
\end{equation}
Combining \eqref{eq:case1-phi} and \eqref{eq:case2-phi}, we obtain
\[
    \phi_0(\mu)
    =
    \begin{cases}
        0, & \mu\ge-\dfrac12,\\[2mm]
        -\infty, & \mu<-\dfrac12.
    \end{cases}
\]
Consequently, the Lagrange dual problem is
    $\sup_{\mu\in\mathbb R} q(\mu;\beta)
    =
    \sup_{\mu\in\mathbb R}
    \{\phi_0(\mu)-\mu\beta\}
    =
    \sup_{\mu\ge-\frac12}
    \{-\mu\beta\}$.
Since \(\beta>0\), the function \(\mu\mapsto-\mu\beta\) is strictly decreasing.
Therefore its supremum over \([-\frac12,\infty)\) is attained at the left
endpoint 
    $\mu^*
    =
    -\frac12$.
The dual value is
\[
    \sup_{\mu\ge-\frac12}\{-\mu\beta\}
    =
    \frac12\beta.
\]
Comparing with the primal value \eqref{eq:degenerate-primal-value}, we get
    $V_{\rm LQ}(\beta)
    =
    \frac12\beta
    =
    \sup_{\mu\in\mathbb R}
    q(\mu;\beta)$,
so there is no duality gap.
Finally, the envelope identity gives
\begin{equation}\label{eq:degenerate-envelope}
    V_{\rm LQ}'(\beta)
    =
    \frac12
    =
    -\mu^*.
\end{equation}
Indeed, since
    $V_{\rm LQ}(\beta)=\frac12\beta$,
we have
    $\frac{d}{d\beta}V_{\rm LQ}(\beta)
    =
    \frac12$.
On the other hand, \(\mu^*=-\frac12\), and therefore
    $-\mu^*
    =
    V_{\rm LQ}'(\beta)$.
The degeneracy is visible in the modified control weight. In this scalar
example \(R=D=1\), hence
    $R+2\mu^*D
    =
    1+2\left(-\frac12\right)
    =
    0$.
Thus the optimal multiplier lies at the boundary of the dual feasibility
region. The sensitivity formula \eqref{eq:degenerate-envelope} remains valid,
but the regular Riccati feedback formula, which requires
\(R+2\mu^*D\succ0\), is not applicable.

\begin{remark}[Interpretation]\label{rem:degenerate-interpretation}
Polyak's certificate $\exists\mu:R+2\mu D\succ0$ \emph{is} satisfied in the above example (take $\mu=0$, giving $R=1\succ0$), so
hidden convexity and strong duality hold. The degeneracy is that the
\emph{optimal} multiplier $\mu^*=-\tfrac12$ sits on the boundary
$R+2\mu^*D=0$. This is not a defect of the theory but its sharp edge: imposing an
equality directly on the same quadratic quantity that appears in the objective
fixes the objective value once $\beta$ is fixed, and the multiplier then cancels
the control weight in stationarity. The example therefore separates the two roles
of $R+2\mu D$---convexity certificate (needs \emph{some} $\mu$) versus synthesis
regularity (needs the \emph{optimal} $\mu^*$)---which the regular Riccati theory
conflates.
\end{remark}

\section{Numerical Realization and Conclusions}
\label{sec:numerics}

The sensitivity results of Sections~\ref{sec:sensitivity}--\ref{sec:lq} are
directly testable: one computes the isoperimetric multiplier $\mu$, verifies the
integral constraint, and checks the envelope law $V'(B)=-\mu$ against a
finite-difference slope of the value function. This section gives a compact,
reproducible algorithm for the computation and reports three experiments---one
in the structured class, one in the regular linear-quadratic regime, and one in
the degenerate regime---chosen so that $\mu$ and $V$ are available in closed
form and the numerics can be validated to machine precision.

\subsection{Shooting formulation for the isoperimetric BVP}

Fix a constraint level $B$. Given the Hamiltonian selector
\begin{equation}\label{eq:selector}
  u^{*}(t,x,\lambda,\mu)\in\argmax_{u\in U}\Ht(t,x,u,\lambda,\mu),
\end{equation}
and an unknown pair $(\lambda_0,\mu)\in\Rn\times\R$ consisting of the initial
adjoint $\lambda(0)=\lambda_0$ and the constant isoperimetric multiplier $\mu$,
integrate the augmented state--adjoint system by a standard initial-value
solver \cite{hairer1993solving,yu2026chemotactic,wang2025hybrid,wang2025well,yu2026diagnostic} within an outer shooting iteration
\cite{stoer1980introduction,keller2018numerical,yu2026controlling,ascher1995numerical,betts2010practical,pesch1994practical,yu2026optimization1,yu2026microscopic,trelat2012optimal}
\begin{equation}\label{eq:shooting-ivp}
  \dot x=g\bigl(t,x,u^{*}\bigr),
  \qquad
  \dot\lambda=-\nabla_x\Ht\bigl(t,x,u^{*},\lambda,\mu\bigr),
  \qquad
  \dot z=h\bigl(t,x,u^{*}\bigr),
\end{equation}
with $x(0)=x_0$ and $z(0)=0$. The two-point boundary conditions
\eqref{eq:iso-pmp-terminal}--\eqref{eq:iso-constraint-satisfied} become the
\emph{shooting residual}
\begin{equation}\label{eq:residual}
  \mathcal R(\lambda_0,\mu;B)
  :=
  \begin{pmatrix}
    \lambda(T)-\nabla\phi\bigl(x(T)\bigr)\\[2pt]
    z(T)-B
  \end{pmatrix}
  \in\Rn\times\R,
\end{equation}
and the extremal is characterized by
\begin{equation}\label{eq:residual-zero}
  \mathcal R(\lambda_0,\mu;B)=0 .
\end{equation}
The first block enforces transversality; the second enforces the isoperimetric
constraint $\Cc(u^{*})=z(T)=B$. Once \eqref{eq:residual-zero} is solved, $\mu$ is
the isoperimetric multiplier and, by Theorem~\ref{thm:pmp-representation},
$-\mu\in\partial^{+}V(B)$.
For minimization problems, such as the LQ examples in Section~\ref{sec:lq},
\(\argmax\) in \eqref{eq:selector} is replaced by \(\argmin\).

\begin{proposition}[Local convergence of shooting]\label{prop:shooting}
Assume that on a neighborhood of an extremal the selector
$u^{*}(t,x,\lambda,\mu)$ of \eqref{eq:selector} is single-valued and $C^{1}$ in
$(x,\lambda,\mu)$, that the residual $\mathcal R(\cdot,\cdot;B)$ is $C^{1}$ with
locally Lipschitz derivative, and that the Jacobian
\begin{equation}\label{eq:shooting-jacobian}
  D_{(\lambda_0,\mu)}\mathcal R(\lambda_0^{*},\mu^{*};B)
  \in\R^{(n+1)\times(n+1)}
\end{equation}
is nonsingular at a solution $(\lambda_0^{*},\mu^{*})$. Then Newton's method
\[
  \begin{pmatrix}\lambda_0\\\mu\end{pmatrix}^{\!+}
  =
  \begin{pmatrix}\lambda_0\\\mu\end{pmatrix}
  -\bigl[D_{(\lambda_0,\mu)}\mathcal R\bigr]^{-1}\,
  \mathcal R(\lambda_0,\mu;B)
\]
converges locally quadratically to $(\lambda_0^{*},\mu^{*})$.
\end{proposition}

\begin{proof}
By the smooth dependence of solutions of \eqref{eq:shooting-ivp} on initial
data and parameters (Assumption~\ref{ass:data} and the $C^{1}$ selector), the
map $(\lambda_0,\mu)\mapsto(x(\cdot),\lambda(\cdot),z(\cdot))$ is $C^{1}$, hence
so is $\mathcal R$, with locally Lipschitz derivative by hypothesis. The claim
is then the standard local quadratic convergence of Newton's method for a
$C^{1}$ map with locally Lipschitz Jacobian that is nonsingular at the root
\cite{ortega2000iterative,deuflhard2011newton,wang2026elliptic,nocedal2006numerical,yu2026optimization2,yu2026optimization3}. If the data are merely $C^{1}$
(Jacobian continuous but not Lipschitz), the same argument gives local
superlinear convergence.
\end{proof}

The Jacobian \eqref{eq:shooting-jacobian} is obtained by integrating the
variational system of \eqref{eq:shooting-ivp} with respect to $(\lambda_0,\mu)$
alongside the state, or, for the scalar experiments below, by a forward
difference. We summarize the procedure.

\begin{algorithm}[h]
\caption{Isoperimetric shooting}
\label{alg:shooting_old}
\begin{algorithmic}[1]
    \Require level $B$; initial guess $(\lambda_0,\mu)$; tolerance $\texttt{tol}$.
    \Ensure $(\lambda_0^{*},\mu^{*})$; the multiplier $\mu^{*}$ and the value $V(B)=J(u^{*})$, with $-\mu^{*}\in\partial^{+}V(B)$.
    
    \Repeat
        \State Integrate \eqref{eq:shooting-ivp} on $[0,T]$ with the selector \eqref{eq:selector} to obtain $\bigl(x(\cdot),\lambda(\cdot),z(\cdot)\bigr)$.
        \State Evaluate the residual $\mathcal R(\lambda_0,\mu;B)$ of \eqref{eq:residual}.
        \If{$\|\mathcal R\| \ge \texttt{tol}$}
            \State Form $D_{(\lambda_0,\mu)}\mathcal R$ by the variational system (or a forward difference), and take the Newton step $(\lambda_0,\mu)\leftarrow(\lambda_0,\mu)-[D_{(\lambda_0,\mu)}\mathcal R]^{-1} \mathcal R$.
        \EndIf
    \Until{$\|\mathcal R\| < \texttt{tol}$}
\end{algorithmic}
\end{algorithm}

\subsection{Experiment 1: a transparent scalar benchmark}

Consider the scalar problem with linear dynamics and concave payoff
\[
  \dot x=u,\quad x(0)=x_0,\qquad
  \max\ J(u)=\int_0^T\!\Bigl(a\,x(t)-\tfrac12 u(t)^2\Bigr)dt,\qquad
  \int_0^T u(t)\,dt=B,
\]
which lies in the structured class of Assumption~\ref{ass:structured}
(linear dynamics, concave running payoff, affine constraint integrand $h=u$).
The augmented Hamiltonian is
$\Ht=a x-\tfrac12u^2+(\lambda+\mu)u$, so
$u^{*}=\lambda+\mu$, $\dot\lambda=-a$, and $\lambda(T)=0$; solving the residual
\eqref{eq:residual} gives $\lambda(t)=a(T-t)$ and
\[
  \mu(B)=\frac{B}{T}-\frac{aT}{2},
  \qquad
  V(B)=a x_0 T+\frac{a^2T^3}{6}-\frac{T}{2}\,\mu(B)^2 .
\]
Because the reduced problem is in the structured class, the residual
\eqref{eq:residual} is \emph{affine} in $(\lambda_0,\mu)$, so
Algorithm~\ref{alg:shooting_old} (Newton) converges in a single step---a machine-precision
consistency check of the envelope law. With $a=T=1$, $x_0=0$ and $B=1$ we obtain
$\mu=\tfrac12$, hence $-\mu=-\tfrac12$; Table~\ref{tab:exp1} confirms
$[V(B+\Delta B)-V(B)]/\Delta B\to-\mu$ and that the central difference equals
$-\mu$ exactly (as it must, $V$ being quadratic).

\begin{table}[t]
\centering\small
\caption{Experiment 1: forward and central difference quotients of $V$ at
$B=1$ against $-\mu=-\tfrac12$.}
\label{tab:exp1}
\begin{tabular}{cccc}
\toprule
$\Delta B$ & $\dfrac{V(B+\Delta B)-V(B)}{\Delta B}$
           & $\dfrac{V(B+\Delta B)-V(B-\Delta B)}{2\Delta B}$ & $-\mu$\\
\midrule
$10^{-1}$ & $-0.5500000000$ & $-0.5000000000$ & $-0.5$\\
$10^{-2}$ & $-0.5050000000$ & $-0.5000000000$ & $-0.5$\\
$10^{-3}$ & $-0.5005000000$ & $-0.5000000000$ & $-0.5$\\
$10^{-4}$ & $-0.5000500000$ & $-0.5000000000$ & $-0.5$\\
\bottomrule
\end{tabular}
\end{table}

\subsection{Experiment 2: a regular constrained LQ problem}

Take the scalar minimization problem
\begin{equation}\label{eq:exp2}
  \dot x=u,\quad x(0)=1,\ T=1,\qquad
  \min\ \tfrac12 x(T)^2+\tfrac12\int_0^1 u(t)^2\,dt,\qquad
  \int_0^1 u(t)^2\,dt=\beta ,
\end{equation}
so that $R=D=1$, $Q=0$, $M=1$. 
On the regular branch \(0<\beta<1\), solving the extremal system gives the constant
control $u^{*}=-1/\bigl(2(1+\mu)\bigr)$ and the closed forms
\[
  \mu^{*}(\beta)=\frac{1}{2\sqrt{\beta}}-1,
  \qquad
  V_{\rm LQ}(\beta)=\tfrac12\bigl(1-\sqrt{\beta}\bigr)^2+\tfrac12\beta,
  \qquad
  R+2\mu^{*}D=\frac{1}{\sqrt{\beta}}-1 .
\]
For \(\beta\ge1\), additional zero-mean controls can spend excess energy
without changing \(x(T)\), so the regular one-dimensional branch above no
longer describes the full constrained minimizer. We therefore use
\(0<\beta<1\) in the numerical comparison.

The modified control weight is positive for $\beta\in(0,1)$, so this is the
regular Riccati regime of Section~\ref{sec:lq}; it degenerates to
$R+2\mu^{*}D=0$ exactly as $\beta\uparrow1$, continuously connecting to
Experiment~3. Table~\ref{tab:exp2} and Figure~\ref{fig:exp2} show that
$-\mu^{*}(\beta)$ and the central-difference slope of $V_{\rm LQ}$ agree to
within $1.1\times10^{-6}$ across the grid, confirming $V_{\rm LQ}'(\beta)
=-\mu^{*}(\beta)$ (Theorem~\ref{thm:lq-hidden}(iii)).

\begin{table}[t]
\centering\small
\caption{Experiment 2: value, multiplier slope, central difference, and modified
control weight for problem \eqref{eq:exp2} ($\Delta=10^{-4}$).}
\label{tab:exp2}
\begin{tabular}{ccccc}
\toprule
$\beta$ & $V_{\rm LQ}(\beta)$ & $-\mu^{*}(\beta)$
        & $\dfrac{V_{\rm LQ}(\beta+\Delta)-V_{\rm LQ}(\beta-\Delta)}{2\Delta}$
        & $R+2\mu^{*}D$\\
\midrule
$0.10$ & $0.28377223$ & $-0.58113883$ & $-0.58113903$ & $2.162278$\\
$0.20$ & $0.25278640$ & $-0.11803399$ & $-0.11803402$ & $1.236068$\\
$0.30$ & $0.25227744$ & $\phantom{-}0.08712907$ & $\phantom{-}0.08712906$ & $0.825742$\\
$0.40$ & $0.26754447$ & $\phantom{-}0.20943058$ & $\phantom{-}0.20943058$ & $0.581139$\\
$0.50$ & $0.29289322$ & $\phantom{-}0.29289322$ & $\phantom{-}0.29289322$ & $0.414214$\\
$0.60$ & $0.32540333$ & $\phantom{-}0.35450278$ & $\phantom{-}0.35450277$ & $0.290994$\\
$0.70$ & $0.36333997$ & $\phantom{-}0.40238570$ & $\phantom{-}0.40238569$ & $0.195229$\\
$0.80$ & $0.40557281$ & $\phantom{-}0.44098301$ & $\phantom{-}0.44098300$ & $0.118034$\\
$0.90$ & $0.45131670$ & $\phantom{-}0.47295372$ & $\phantom{-}0.47295372$ & $0.054093$\\
\bottomrule
\end{tabular}
\end{table}

\begin{figure}[htbp]
\centering
\includegraphics[width=\textwidth]{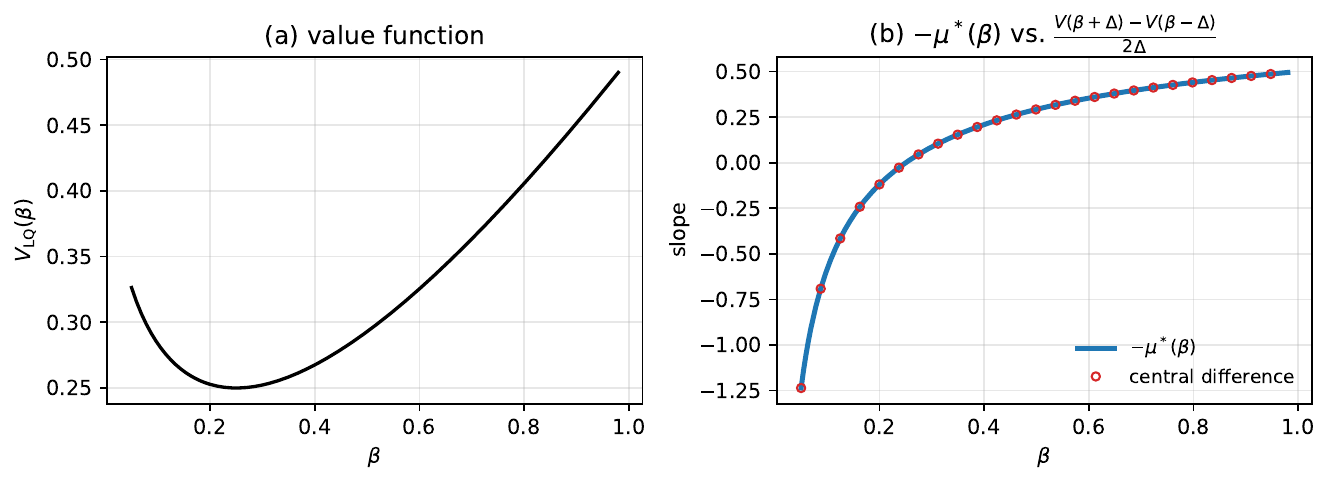}
\caption{Experiment 2. (a) The constrained value function
$V_{\rm LQ}(\beta)$ for problem \eqref{eq:exp2}. (b) The multiplier curve
$-\mu^{*}(\beta)$ (solid) against the central-difference slope of
$V_{\rm LQ}$ (markers); the curves coincide to $\sim10^{-6}$, confirming the
envelope identity $V_{\rm LQ}'(\beta)=-\mu^{*}(\beta)$.}
\label{fig:exp2}
\end{figure}
\clearpage

Although the inner control is linear in the adjoint, the scalar multiplier
equation $\psi(\mu):=\Cc(u_\mu)=\beta$ of Lemma~\ref{lem:secular} is
\emph{nonlinear} in $\mu$, so it exercises the Newton iteration of
Proposition~\ref{prop:shooting}. Table~\ref{tab:newton} reports the iterates for
$\beta=0.25$ (exact root $\mu^{*}=0$): the error is squared at each step,
exhibiting the predicted local quadratic convergence.

\begin{table}[t]
\centering\small
\caption{Experiment 2: Newton iterates for the multiplier equation
$\psi(\mu)=\beta$ at $\beta=0.25$, with exact root $\mu^{*}=0$.}
\label{tab:newton}
\begin{tabular}{cccc}
\toprule
$k$ & $\mu_k$ & $|\psi(\mu_k)-\beta|$ & $|\mu_k-\mu^{*}|$\\
\midrule
$0$ & $\phantom{-}0.200000000000$ & $7.64\times10^{-2}$ & $2.00\times10^{-1}$\\
$1$ & $-0.064000000000$ & $3.54\times10^{-2}$ & $6.40\times10^{-2}$\\
$2$ & $-0.006012928000$ & $3.03\times10^{-3}$ & $6.01\times10^{-3}$\\
$3$ & $-0.000054124255$ & $2.71\times10^{-5}$ & $5.41\times10^{-5}$\\
$4$ & $-0.000000004394$ & $2.20\times10^{-9}$ & $4.39\times10^{-9}$\\
$5$ & $\phantom{-}0.000000000000$ & $0$ & $3.8\times10^{-17}$\\
\bottomrule
\end{tabular}
\end{table}

\subsection{Experiment 3: the degenerate scalar LQ case}

Finally, take $R=D=1$, $Q=M=0$, $A=0$, $B=1$ in \eqref{eq:lq-iso-problem}:
$\min\tfrac12\int_0^T u^2\,dt$ subject to $\int_0^T u^2\,dt=\beta$. Every
feasible control has the same cost, so
  $V_{\rm LQ}(\beta)=\tfrac12\beta$,
  $-\mu^{*}=\tfrac12$,
  $R+2\mu^{*}D=0$.
The envelope law $V_{\rm LQ}'(\beta)=-\mu^{*}=\tfrac12$ holds for every
$\beta>0$ (Table~\ref{tab:exp3}), yet the modified control weight vanishes: this
is the trust-region hard case of Lemma~\ref{lem:secular}, in which the Riccati
feedback \eqref{eq:lq-control-mu}--\eqref{eq:lq-modified-riccati} is undefined
even though sensitivity is intact. Numerically, Algorithm~\ref{alg:shooting_old} detects the
degeneracy through the singularity of the modified weight $R+2\mu D$ at the
optimal multiplier; the multiplier itself, and hence $V_{\rm LQ}'(\beta)$,
remain well defined. This experiment is the operational reason the Riccati
representation of Section~\ref{sec:lq} must carry the regularity condition
$R+2\mu^{*}D\succ0$ as a genuine hypothesis rather than a convenience.

\begin{table}[t]
\centering\small
\caption{Experiment 3: the degenerate case. Sensitivity survives while the
modified control weight is singular.}
\label{tab:exp3}
\begin{tabular}{cccc}
\toprule
$\beta$ & $V_{\rm LQ}(\beta)$ & $-\mu^{*}$ & $R+2\mu^{*}D$\\
\midrule
$0.25$ & $0.125$ & $0.5$ & $0$\\
$0.50$ & $0.250$ & $0.5$ & $0$\\
$1.00$ & $0.500$ & $0.5$ & $0$\\
$2.00$ & $1.000$ & $0.5$ & $0$\\
\bottomrule
\end{tabular}
\end{table}

\subsection{Conclusions}

We have developed a function-space duality framework for sensitivity analysis
in finite-horizon optimal control problems with scalar isoperimetric
constraints. The main conclusion is that the isoperimetric Pontryagin
multiplier \(\mu\) has a shadow-price interpretation only after the
constraint-level value function has been placed in an appropriate convex
duality setting. More precisely, under the Lagrangian convention
    $J(u)+\mu\bigl(\Cc(u)-B\bigr)$,
the marginal value of the isoperimetric level is \(-\mu\), not \(\mu\). Thus,
whenever the constrained value function is differentiable at \(B\),
    $V'(B)=-\mu$.
At nondifferentiability points the corresponding statement is the
superdifferential inclusion
    $-\mu\in\partial^+V(B)$.

The analysis shows that this sensitivity formula is fundamentally a
duality result in function spaces, not a consequence of the maximum principle
alone. In the structured concave-affine regime, the admissible controls form a
weak-\(*\) compact subset of \(L^\infty(0,T;\mathbb R^m)\), the linear state
equation induces an affine control-to-state map into
\(W^{1,\infty}(0,T;\mathbb R^n)\), and the affine isoperimetric functional
maps the admissible set onto an interval. These properties imply concavity of
the constrained value function and, under a two-sided attainability condition,
yield the Fenchel--Moreau representation
\[
    V(B)=\inf_{\mu\in\mathbb R}
    \bigl\{\ell(\mu)-\mu B\bigr\},
    \qquad
    \ell(\mu)=\sup_{u\in\U}\{J(u)+\mu\Cc(u)\}.
\]
Consequently, the identity
    $\partial^+V(B)=-\mathcal D(B)$
is a direct conjugate-duality statement, where \(\mathcal D(B)\) denotes the
set of dual-optimal multipliers. The remaining step is not convex analytic but
variational: a constraint qualification is required to identify
\(\mathcal D(B)\) with the set \(\mathcal M(B)\) of normal Pontryagin
isoperimetric multipliers. This separates the free duality part of the
argument from the genuine normality and multiplier-representation part.

For the linear-quadratic problem with a single quadratic equality constraint,
the state equation reduces the problem to quadratic forms on the Hilbert space
\(L^2(0,T;\mathbb R^m)\):
\[
    J_{\rm LQ}(u)
    =
    \frac12\langle u,\mathcal R u\rangle
    +\langle \rho_0,u\rangle+\gamma_0,
    \qquad
    \Cc(u)=\langle u,\mathcal D u\rangle .
\]
This exposes a second hidden-convexity mechanism. Although the equality
constraint \(\Cc(u)=\beta\) is nonconvex, strong duality can still be obtained
through the convexity of the joint range of the relevant quadratic
functionals. The multiplier sensitivity formula remains valid at
differentiability points of \(V_{\rm LQ}\), while Riccati synthesis requires
the additional pointwise regularity condition
    $R(t)+2\mu^*D(t)\succ0$ for a.e. $t\in[0,T]$.
The degenerate scalar example shows that this condition is genuinely stronger
than sensitivity: at the boundary case \(R+2\mu^*D=0\), the envelope identity
\(V_{\rm LQ}'(\beta)=-\mu^*\) remains valid, but the modified Riccati feedback
formula is no longer defined.

The numerical shooting procedure provides a practical realization of the
theory. It computes the constant isoperimetric multiplier as part of a
two-point boundary-value problem and verifies the envelope relation by
comparing \(-\mu\) with finite-difference slopes of the constrained value
function. The structured scalar benchmark, the regular LQ example, and the
degenerate LQ example collectively confirm the distinction between
value-function sensitivity and feedback synthesis.

Several extensions remain natural. First, vector-valued isoperimetric
constraints \(\Cc(u)=B\in\mathbb R^k\) would replace the scalar multiplier by a
vector multiplier and the one-dimensional superdifferential by a convex subset
of \(\mathbb R^k\). The Fenchel-duality part extends directly, but the hidden
convexity of quadratic ranges becomes substantially more restrictive when more
than two quadratic functionals are involved. Second, inequality constraints
\(\Cc(u)\le B\) introduce the sign condition \(\mu\ge0\) and complementary
slackness, with the relevant constraint qualification becoming an inequality
version of the Zowe--Kurcyusz or Robinson regularity condition. Third, for
genuinely nonlinear dynamics outside the structured concave-affine class, one
should expect weak compactness and strong duality only after relaxation or
convexification; the sensitivity theorem then applies first to the relaxed
value function, while the gap between the relaxed and original problems becomes
a separate analytical question.

The central message is that the shadow-price interpretation of an
isoperimetric Pontryagin multiplier is exactly as strong as the function-space
duality structure supporting it. Weak compactness, affine or hidden convexity,
Fenchel--Moreau conjugacy, and normal multiplier representation are the
mechanisms that turn a first-order multiplier into a genuine derivative or
supergradient of the constrained value function.

\appendix
\section*{Appendix}
\section{Closed-form verification of the envelope identity}
\label{subsec:closed-form-experiments}
We verify the identity
    $V'(B)=-\mu$
in two scalar benchmark problems. The first is a maximization problem in the
structured affine-constraint class of Section~\ref{sec:sensitivity}. The second
is a minimization-form linear-quadratic problem with a quadratic equality
constraint. In both cases the same Lagrangian convention is used:
    $L(u,\mu;B)=J(u)+\mu(\mathcal C(u)-B)$,
so the right-hand side \(B\) enters the Lagrangian through the term
\(-\mu B\). Therefore, whenever the value function is differentiable,
\[
    \frac{dV}{dB}
    =
    \frac{\partial L}{\partial B}(u^*,\mu^*;B)
    =
    -\mu^*.
\]
\subsection{Experiment 1: affine constraint and concave payoff}
Consider the scalar maximization problem
\begin{equation}\label{eq:exp1-main}
\begin{aligned}
    &\text{maximize} &&
    J(u)
    :=
    \int_0^T
    \left[
        a\,x(t)-\frac12 u(t)^2
    \right]dt
    \\
    &\text{subject to} &&
    \dot x(t)=u(t),
    \qquad
    x(0)=x_0,
    \\
    &&&
    \mathcal C(u):=\int_0^T u(t)\,dt=B .
\end{aligned}
\end{equation}
The augmented Hamiltonian is
    $\widetilde H(t,x,u,\lambda,\mu)
    =
    a x-\frac12 u^2+\lambda u+\mu u$.
The stationarity condition is
    $\partial_u\widetilde H
    =
    -u+\lambda+\mu
    =
    0$,
hence
    $u^*(t)=\lambda(t)+\mu$.
The adjoint equation and terminal condition are
\begin{equation}\label{eq:exp1-adjoint}
    \dot\lambda(t)
    =
    -\partial_x\widetilde H(t,x,u,\lambda,\mu)
    =
    -a,
    \qquad
    \lambda(T)=0 .
\end{equation}
Solving \eqref{eq:exp1-adjoint} gives
    $\lambda(t)
    =
    a(T-t)$.
Therefore
    $u^*(t)
    =
    a(T-t)+\mu$.
The constraint determines \(\mu\). Indeed,
\[
    B=
    \int_0^T u^*(t)\,dt=
    \frac{aT^2}{2}+\mu T .
\]
Thus
\begin{equation}\label{eq:exp1-mu}
    \mu(B)
    =
    \frac{B}{T}-\frac{aT}{2}.
\end{equation}
Next compute the state:
\[
    x^*(t)=
    x_0+\int_0^t u^*(s)\,ds=
    x_0+\int_0^t
    \left[
        a(T-s)+\mu
    \right]ds=
    x_0+aTt-\frac{a}{2}t^2+\mu t .
\]
Now evaluate the value function. Since
\[
    u^*(t)=a(T-t)+\mu,
    \qquad
    x^*(t)=x_0+aTt-\frac{a}{2}t^2+\mu t,
\]
we have
\[
    V(B) =
    J(u^*)=
    \int_0^T
    \left[
        a x^*(t)-\frac12 u^*(t)^2
    \right]dt=
    \int_0^T
    \left[
        a\left(
            x_0+aTt-\frac{a}{2}t^2+\mu t
        \right)
        -
        \frac12\left(
            a(T-t)+\mu
        \right)^2
    \right]dt .
\]
Expanding the integrand,
\[
    a x^*(t)-\frac12u^*(t)^2=
    a x_0
    -
    \frac{a^2T^2}{2}
    -
    a\mu T
    -
    \frac12\mu^2
    +
    2a^2Tt
    +
    2a\mu t
    -
    a^2t^2 .
\]
Integrating term by term gives
\[
    V(B)
    =
    \int_0^T
    \left[
        a x_0
        -
        \frac{a^2T^2}{2}
        -
        a\mu T
        -
        \frac12\mu^2
        +
        2a^2Tt
        +
        2a\mu t
        -
        a^2t^2
    \right]dt=
    a x_0T+\frac{a^2T^3}{6}-\frac{T}{2}\mu^2 .
\]
Substituting \(\mu=\mu(B)\) from \eqref{eq:exp1-mu}, we obtain the closed-form
value function
\begin{equation}\label{eq:exp1-value-final}
    V(B)
    =
    a x_0T+\frac{a^2T^3}{6}
    -
    \frac{T}{2}
    \left(
        \frac{B}{T}-\frac{aT}{2}
    \right)^2 .
\end{equation}
Differentiate \eqref{eq:exp1-value-final} with respect to \(B\). Since
    $\mu(B)=\frac{B}{T}-\frac{aT}{2},
    \qquad
    \mu'(B)=\frac1T$,
we have
\[
    V'(B)
    =
    \frac{d}{dB}
    \left[
        a x_0T+\frac{a^2T^3}{6}
        -
        \frac{T}{2}\mu(B)^2
    \right]=
    -T\mu(B)\mu'(B)=
    -\mu(B).
\]
As a concrete numerical specialization, if \(a=T=1\) and \(x_0=0\), then
\[
    \mu(B)=B-\frac12,
    \qquad
    V(B)=\frac16-\frac12\left(B-\frac12\right)^2,
    \qquad
    V'(B)=-\left(B-\frac12\right).
\]
At \(B=1\),
    $\mu(1)=\frac12$,
    $V'(1)=-\frac12=-\mu(1)$.

\subsection{Experiment 2: regular scalar LQ problem}
Consider the scalar minimization problem
\[
\begin{aligned}
    &\text{minimize} &&
    J_{\rm LQ}(u)
    :=
    \frac12 x(1)^2
    +
    \frac12\int_0^1 u(t)^2\,dt
    \\
    &\text{subject to} &&
    \dot x(t)=u(t),
    \qquad
    x(0)=1,
    \\
    &&&
    \mathcal C(u):=\int_0^1 u(t)^2\,dt=\beta .
\end{aligned}
\]
We focus on the regular branch \(0<\beta<1\). Since
    $x(1)=1+\int_0^1 u(t)\,dt$,
the Cauchy--Schwarz inequality gives
\[
    \left|
        \int_0^1 u(t)\,dt
    \right|^2
    \le
    \left(\int_0^1 1^2\,dt\right)
    \left(\int_0^1 u(t)^2\,dt\right)
    =
    \beta .
\]
Therefore
    $-\sqrt{\beta}
    \le
    \int_0^1 u(t)\,dt
    \le
    \sqrt{\beta}$.
For \(0<\beta<1\), the terminal penalty
\[
    \frac12 x(1)^2
    =
    \frac12
    \left(
        1+\int_0^1 u(t)\,dt
    \right)^2
\]
is minimized by choosing the smallest possible value of \(\int_0^1u(t)\,dt\),
namely
    $\int_0^1 u(t)\,dt=-\sqrt{\beta}$.
Equality in Cauchy--Schwarz occurs if and only if \(u\) is constant a.e.; hence
the optimizer on the regular branch is
    $u^*(t)=-\sqrt{\beta}$,
    $0\le t\le1$.
Then
    $x^*(1)
    =
    1+\int_0^1 u^*(t)\,dt
    =
    1-\sqrt{\beta}$.
The value is therefore
\[
    V_{\rm LQ}(\beta)
    =
    \frac12 x^*(1)^2
    +
    \frac12\int_0^1 u^*(t)^2\,dt
    =
    \frac12(1-\sqrt{\beta})^2
    +
    \frac12\beta .
\]
Expanding gives the useful form
    $V_{\rm LQ}(\beta)
    =
    \frac12-\sqrt{\beta}+\beta$.
Hence
    $V_{\rm LQ}'(\beta)
    =
    -\frac{1}{2\sqrt{\beta}}+1
    =
    1-\frac{1}{2\sqrt{\beta}}$.
We now compute the Lagrange multiplier independently from stationarity. The
minimization-form augmented Hamiltonian is
\[
    \widetilde H(t,x,u,\lambda,\mu)
    =
    \frac12 u^2+\lambda u+\mu u^2 .
\]
The stationarity condition is
\begin{equation}\label{eq:exp2-stationarity}
    \partial_u\widetilde H
    =
    (1+2\mu)u+\lambda
    =
    0.
\end{equation}
The adjoint equation and terminal condition are
    $\dot\lambda(t)
    =
    -\partial_x\widetilde H=0$,
    $\lambda(1)=x(1)$.
Thus \(\lambda(t)\equiv\lambda=x(1)\). Using
\eqref{eq:exp2-stationarity},
\begin{equation}\label{eq:exp2-u-lambda}
    u(t)
    =
    -\frac{\lambda}{1+2\mu}.
\end{equation}
Since \(u\) is constant,
    $x(1)=1+u$.
But \(\lambda=x(1)\), so
    $\lambda=1+u$,
    $u=-\frac{\lambda}{1+2\mu}$.
The constraint gives
    $\beta=\int_0^1u(t)^2\,dt=u^2$.
On the regular branch \(0<\beta<1\), the optimal control is negative, so
    $u=-\sqrt{\beta}$.
Therefore
    $\lambda=1+u=1-\sqrt{\beta}$.
Substituting \(u=-\sqrt{\beta}\) and \(\lambda=1-\sqrt{\beta}\) into the
stationarity condition
    $(1+2\mu)u+\lambda=0$
gives $(1+2\mu)(-\sqrt{\beta})+(1-\sqrt{\beta})=0$.
Thus
\begin{equation}\label{eq:exp2-mu}
    \mu^*(\beta)
    =
    \frac{1}{2\sqrt{\beta}}-1 .
\end{equation}
Consequently,
    $-\mu^*(\beta)
    =
    1-\frac{1}{2\sqrt{\beta}}
    =
    V_{\rm LQ}'(\beta)$.
The regularity matrix in this scalar problem is
\[
    R+2\mu^*D
    =
    1+2\left(
        \frac{1}{2\sqrt{\beta}}-1
    \right)
    =
    \frac{1}{\sqrt{\beta}}-1 .
\]
Therefore,
\[
    R+2\mu^*D>0
    \quad\Longleftrightarrow\quad
    \frac{1}{\sqrt{\beta}}-1>0
    \quad\Longleftrightarrow\quad
    0<\beta<1.
\]
Thus \(0<\beta<1\) is exactly the regular branch on which the Riccati or
stationarity formula using \((R+2\mu^*D)^{-1}\) is valid.
For completeness, the scalar multiplier equation can also be written as a
secular equation. From \eqref{eq:exp2-u-lambda} and the state-adjoint relation
one obtains
    $u_\mu
    =
    -\frac{1}{2(1+\mu)}$.
Hence
    $\psi(\mu)
    :=
    \int_0^1 u_\mu^2\,dt
    =
    \frac{1}{4(1+\mu)^2}$.
The constraint \(\psi(\mu)=\beta\) gives
    $\frac{1}{4(1+\mu)^2}
    =
    \beta$.
Solving for the regular branch \(1+\mu>0\),
\[
    2(1+\mu)\sqrt{\beta}=1,
    \qquad
    \mu=\frac{1}{2\sqrt{\beta}}-1,
\]
which agrees with \eqref{eq:exp2-mu}. Differentiating the secular residual
\[
    F(\mu):=\psi(\mu)-\beta
    =
    \frac{1}{4(1+\mu)^2}-\beta
\]
gives
    $F'(\mu)
    =
    -\frac{1}{2(1+\mu)^3}$.
Newton's method for the multiplier is therefore
\[
    \mu_{k+1}
    =
    \mu_k
    -
    \frac{F(\mu_k)}{F'(\mu_k)}
    =
    \mu_k
    +
    2(1+\mu_k)^3
    \left[
        \frac{1}{4(1+\mu_k)^2}
        -
        \beta
    \right].
\]
Finally, for \(\beta=1/4\),
\[
    \mu^*\!\left(\frac14\right)
    =
    \frac{1}{2(1/2)}-1
    =
    0,
    \qquad
    V_{\rm LQ}'\!\left(\frac14\right)
    =
    1-\frac{1}{2(1/2)}
    =
    0
    =
    -\mu^*\!\left(\frac14\right).
\]

\bibliography{2reference}

\end{document}